\documentclass[review, compress, 3p]{elsarticle}
\usepackage{lineno,hyperref}
\usepackage{amssymb,amsfonts}
\usepackage{amsmath,amsthm}
\usepackage{algorithm,algorithmic}
\usepackage{caption}
\usepackage{graphicx,subfigure}
\usepackage{float,rotating}
\usepackage{booktabs,array}
\usepackage{multirow,multicol}
\usepackage{lscape}
\usepackage{epstopdf}
\usepackage{color}
\usepackage{setspace}
\usepackage{natbib}
\usepackage{bm}
\usepackage{url,doi}
\usepackage{graphicx}
\usepackage{float}
\usepackage{amsmath}
\usepackage{hyperref}

\usepackage[capitalize]{cleveref} 
\usepackage{pgfplots}
\pgfplotsset{compat=1.18}

\renewcommand\figurename{Fig.}

\makeatletter
\@addtoreset{equation}{section}
\makeatother
\renewcommand{\theequation}{\arabic{section}.\arabic{equation}}

\newtheorem{theorem}{Theorem}[section]
\newtheorem{lemma}{Lemma}[section]

\journal{Elsevier}

\begin{document}

\begin{frontmatter}

\title{A direct PinT algorithm for higher-order nonlinear time-evolution equations}

\author[address1]{Shun-Zhi Zhong}
%\cortext[correspondingauthor]{Corresponding authors}
\ead{szzhongm@sina.com}

\author[address1]{Yong-Liang Zhao\corref{correspondingauthor}}
\cortext[correspondingauthor]{Corresponding author}
%\cortext[correspondingauthor]{Corresponding authors}
\ead{ylzhaofde@sina.com}

\author[address1]{Qian-Yu Shu}
%\cortext[correspondingauthor]{Corresponding authors}
\ead{shuqy@sicnu.edu.cn}

%\author[address3]{Meng Li}
%\ead{limeng@zzu.edu.cn}

%\author[address2]{Alexander Ostermann}
%\ead{alexander.ostermann@uibk.ac.at}

\address[address1] {School of Mathematical Sciences, 
	Sichuan Normal University, Chengdu, Sichuan 610066, P.R. China}
%\address[address2]{Department of Mathematics, University of Innsbruck, Technikerstra{\ss}e 13, Innsbruck 6020, Austria}
%\address[address2] {School of Mathematical Sciences, 
%	Sichuan Normal University, Chengdu, Sichuan 610066, P.R. China}
%%\address[address3] {School of Mathematics and Statistics, Zhengzhou University,
%%Zhengzhou, Henan 450001, P.R. China}
%\address[address2] {School of Mathematical Sciences, 
%	Sichuan Normal University, Chengdu, Sichuan 610066, P.R. China}

\begin{abstract}
	Higher-order nonlinear time-evolution equations have widespread applications in science and engineering, such as in solid mechanics, materials science, and fluid mechanics. This paper mainly studies a direct time-parallel algorithm for solving time-dependent differential equations of orders 1 to 3. Different from the traditional time-stepping approach, we directly solve the all-at-once system from higher-order evolution equations by diagonalization the time discretization matrix $B$. Based on the connection between the characteristic equation and Chebyshev polynomials, we give explicit formulas for the eigenvector matrix  $V$ of $B$ and its inverse  $V^{-1}$. We prove that  $Cond_2\left( V \right) =\mathcal{O} \left( n^3 \right)$, where $n$ is the number of time steps. A direct parallel-in-time algorithm is designed by exploring the structure of the spectral decomposition of $B$. Numerical experiments are provided to show the significant computational speedup of the proposed algorithm. 
\end{abstract}

\end{frontmatter}
%\begin{msc}
%	65M55, 65N12, 65M15, 65Y05
%\end{msc}

\section{Introduction}
\label{sec1}
For time evolution problems, parallelization in the time direction has been a highly active research topic in recent years. This is because many time evolution problems cannot be quickly computed to obtain the desired solutions on modern supercomputers. Introducing parallelization can significantly reduce computational cost and accelerates computational efficiency. For further acceleration in the time direction, at least for strongly dissipative problem, the parallel algorithm in \cite{2001R} and many other variants (e.g., the MGRiT algorithm \cite{2015Parallel} and the PFASST algorithm\cite{2012Toward}) perform very well in this aspect. However, for wave propagation problems, since the convergence speed largely depends on dissipativity, the performance of these representative algorithms is not satisfactory. In addition, many studies have improved the convergence behavior of the iterative Parallel-in-Time (PinT) algorithm by modifying the coarse grid correction \cite{2014On,2013Stable,2010Time,2020A,2012Explicit}. As pointed out in \cite{2018Wave}, these modified algorithms either require a large amount of additional computational burden or have very limited applicability.
For detailed studies on time-parallel time integration, see, e.g., \cite{gander2015,GanderLunet2024,2025Time}. In this paper, we focus on the PinT algorithm based on diagonalization techniques, which was first proposed by Maday and Rønquist \cite{2008Parallelization}. For linear ordinary differential  systems with initial value, the idea can be described as follows:
\begin{equation*}
	u^{\prime} \left( t \right) +Au\left( t \right) =f\left( t \right) , \quad u\left( 0 \right) =u_0\in \mathbb{R} ^{m},
\end{equation*}
where $u_0$ and $A\in \mathbb{R} ^{m\times m}$ and $f\left( t \right)$ are known terms. We discretize the time derivative using a finite difference method and formulate it as an all-at-once linear system  
\begin{equation}
	\mathbb{M} \boldsymbol{u}:=\left( B\otimes I_x+I_t\otimes A \right) \boldsymbol{u}=\boldsymbol{b},
	\label{eq1.1:energy}
\end{equation}
where $\boldsymbol{u}=\left[ {u}_1^\top,{u}_2^\top,\cdots ,{u}_n^\top \right] ^\top$ with $u_j\approx u\left( t_j \right) $ ($\{t_j\}_{j=0}^{n}$ is a division of time domain $[0,T]\subset\mathbb{R}$), $\boldsymbol{b}$ contains the initial value and
right-hand-side information, $I_x$ and $I_t$  denote two identity matrices with suitable sizes, respectively. $B\in \mathbb{R} ^{n\times n}$ is the time discretization matrix. If $B$ is diagonalizable, that is, $B=VDV^{-1}$ with $D=\mathrm{diag}\left( \lambda _{1,}\cdots ,\lambda _n \right) $ ($\{\lambda_j\}_{j=1}^{n}$ are the eigenvalue of $B$),  we can factorize $\mathbb{M}$ as 
\begin{equation*}
	\mathbb{M} =\left( V\otimes I_x \right) \left( D\otimes I_x+I_t\otimes A \right) \left( V^{-1}\otimes I_x \right).
\end{equation*}
Thus, \eqref{eq1.1:energy} can be directly solved through the following three steps:
\begin{equation}
	\begin{cases}  
		\boldsymbol{g} = (V^{-1} \otimes I_x) \boldsymbol{b},   & \text{step-(a)}; \\  
		(\lambda_j I_x + A) z_j = g_j, \quad j = 1, 2, \ldots, n, &\text{step-(b)};\\  
		\boldsymbol{u} = (V \otimes I_x) \boldsymbol{z}, & \text{step-(c)};  
	\end{cases}
	\label{eq1.2:energy}
\end{equation}
where $\boldsymbol{w}=\left( {w}_1^\top,{w}_2^\top,\cdots ,{w}_n^\top \right) ^\top$ and $\boldsymbol{g}=\left( {g}_1^\top,{g}_2^\top,\cdots ,{g}_n^\top \right) ^\top$. Obviously, in \eqref{eq1.2:energy}, since the first and third steps require only to performs matrix-vector multiplications, they are inherently parallelizable. The main computational cost comes from the second step. Since these complex linear systems are completely decoupled, it can be solved by some direct or iterative solvers in a parallel pattern.

The core lies in how to accurately and efficiently diagonalize $B$. Using the conventional backward Euler discretization with uniform time steps obviously results in a non-diagonalizable matrix. To achieve diagonalization, different time steps $\left\{ \varDelta t_j \right\} _{j=1}^{n}$ were used in  \cite{2008Parallelization}, which leads to 
\begin{equation*}
	B= \begin{bmatrix}
		\frac{1}{\varDelta t_1}&		&		&		\\
		-\frac{1}{\varDelta t_2}&		\frac{1}{\varDelta t_2}&		&		\\
		&		\ddots&		\ddots&		\\
		&		&		-\frac{1}{\varDelta t_n}&		\frac{1}{\varDelta t_n}\\
	\end{bmatrix}.
\end{equation*}
Clearly, the matrix $B$ is diagonalizable. However, in practical applications, the condition number of its eigenvector matrix  $V$  can be very large, which increases the roundoff errors in step-(a) and step-(c) of \eqref{eq1.2:energy}. This seriously affects the accuracy of the numerical solutions, see \cite{2019A} for details. 
%\begin{equation*}
%	\mathrm{roundoff\ error}=\mathcal{O} \left( \epsilon Cond_2\left( V \right) \right) ,
%\end{equation*}
%where $\epsilon$ is the machine precision. 
They \cite{2019A} considered geometrically increasing step sizes $\left\{ \varDelta t_j=\varDelta t_1\tau ^{n-1} \right\} _{j=1}^{n}$, and through this choice matrix $B$ can be diagonalized, where $\tau >1$ is a parameter. However, the selection of $\tau$ significantly impacts the error, making it particularly challenging to choose an appropriate $\tau$. Numerical results in \cite{2019A} show that $n$ can only be approximately between 20 and 25, to balance rounding errors and discretization errors. Therefore, parallelism is limited for larger $n$. In \cite{2022A}, the authors studied first-order and special second-order partial differential equations (PDEs). In this work, we eliminate the undesirable restriction on  $n$ \cite{2019A}, and extend the special second-order PDEs \cite{2022A} to more general form of first- to third-order PDEs. For time discretization, a central finite differences formula is used in the first $n-1$ steps, and a BDF2 formula is applied in the last step, that is, 
\begin{equation}
	\begin{cases}
		\frac{u_{j+1}-u_{j-1}}{2\varDelta t}+Au_j=f_j, \quad j=1,2,\cdots ,n-1,\\
		\frac{\frac{3}{2}u_n-2u_{n-1}+\frac{1}{2}u_{n-2}}{\varDelta t}+Au_n=f_n,          \\
	\end{cases} 
	\label{eq1.3:energy}
\end{equation}
where $\varDelta t=T/n$ and $f_i=f(t_i) (i=1, \cdots, n)$. The all-at-once system of \eqref{eq1.3:energy} has the form of \eqref{eq1.1:energy}, where $B$ and $\boldsymbol{b}$ are given as follows: 
\begin{equation}
	B=\frac{1}{\varDelta t} \begin{bmatrix}
		0&		\frac{1}{2}&		&		&		\\
		-\frac{1}{2}&		0&		\frac{1}{2}&		&		\\
		&		\ddots&		\ddots&		\ddots&		\\
		&		&		-\frac{1}{2}&		0&		\frac{1}{2}\\
		&		&		\frac{1}{2}&		-2&		\frac{3}{2}\\
	\end{bmatrix}  , \quad \boldsymbol{b}=\left[ \begin{array}{c}
		\frac{u_0}{2\varDelta t}+f_1\\
		f_2\\
		\vdots\\
		f_n\\
	\end{array} \right] .
	% \boldsymbol{u}=\left[ \begin{array}{c}
		%	u_1\\
		%	u_2\\
		%	\vdots\\
		%	u_n\\
		%\end{array} \right] .
		\label{eq1.4:energy}
	\end{equation}
	We mention that there are other diagonalization-based PinT algorithms, which use some preprocessing techniques to handle the all-at-once system \eqref{eq1.1:energy} and perform well for larger  $n$; see \cite{2021A,2022Space,2019note,2022A,2020AFAST,2018Preconditioning,2021Matrix}. As a prominent case of boundary value methods (BVMs)\cite{1998Solving}, the hybrid time discretization scheme \eqref{eq1.3:energy} is frequently employed in computational science and engineering.
	
	Inspired by the pioneering works of \cite{2022A, 2008Parallelization}, we attempt to design a direct PinT algorithm for solving the all-at-once system of \eqref{eq1.3:energy} (rather than iteratively solving it, e.g., \cite{1993Parallel}). Subsequently, we extend this algorithm to solve second- and third-order nonlinear differential equations.
	For first- and second-order differential equations, we also notice some relevant fast time-domain algorithms combined with this algorithm, 
	such as the space-time discretization method \cite{2014Space,2015Multilevel}, 
	the low-rank approximation technique \cite{2011Low,2015Matlab}, the domain decomposition algorithm \cite{2015Domain}, 
	and the inversion algorithm \cite{2022A}. For further details on all-at-once systems, refer to \cite{2020parallel-in-timeiterativealgorithm,lin2024parallel,wu2023pint,wu2022uniform,2023bilateralpreconditioning,2024parallelpreconditioner}. 
	The main contributions of this work  are summarized as folllows:
	
	(I) We theoretically prove that the matrix $B$ in \eqref{eq1.4:energy} can be diagonalized as $B=VDV^{-1}$.
	
	(II) The analytical expressions of $V$, $V^{-1}$ and $D$ are given. We rigorously prove that the condition number of matrix $V$ satisfies $Cond_2\left( V \right) = \mathcal{O} \left( n^3 \right)$. 
	
	(III) A fast algorithm for computing $V^{-1}$ is proposed, which is faster than the MATLAB built-in function $\texttt{eig}$.

The rest of this paper is organized as follows. \autoref{sec2:intro}, our direct PinT algorithm for three nonlinear problems is introduced.
\autoref{sec3} provides analytical expressions of the matrices  $V$  and  $D$. The diagonalizability of $B$ in \eqref{eq1.4:energy} and the condition number of the eigenvector matrix $V$ are studied. 
\autoref{sec4} propose a fast algorithm with complexity  $\mathcal{O} \left( n^2 \right) $  for efficiently computing of $V^{-1}$.
\autoref{sec5} displays some numerical experiments to verify our theoretical results.
\autoref{sec6} concludes this paper.
\section{The PinT algorithm for nonlinear problems} 
\label{sec2:intro}
%\vspace{-3mm}
In this section, we introduce the time discretization and our direct PinT algorithm to solve nonlinear differential equations with general time derivatives of orders 1 to 3.

\subsection{First-order problems}\label{sec2.1}
We first consider the following first-order problem
\begin{equation}
	u' \left( t \right) +f\left( u\left( t \right) \right) =0, \quad u\left( 0 \right) =u_0,
	\label{eq2.1:energy}
\end{equation}
where $t\in \left( 0,T \right] $, $u\left( t \right) \in \mathbb{R} ^m$, and $f: \left[ 0,T \right] \times \mathbb{R} ^m\rightarrow \mathbb{R} ^m$. 
Applying the time discretization \eqref{eq1.3:energy} to \eqref{eq2.1:energy}, the all-at-once system of \eqref{eq2.1:energy} is written as
\begin{equation}
	\left( B\otimes I_x \right) \boldsymbol{u}+F\left( \boldsymbol{u} \right) =\boldsymbol{b},
	\label{eq2.3:energy}
\end{equation}
where $F\left( \boldsymbol{u} \right) =\left( f^\top\left( u_1 \right) ,f^\top\left( u_2 \right) ,\cdots ,f^\top\left( u_n \right) \right) ^\top$ and $\boldsymbol{b}=\left( \frac{{u}_0^\top}{2\varDelta t},0,\cdots ,0 \right) ^\top$. Using the standard Newton's iteration for \eqref{eq2.3:energy} can obtain 
\begin{equation}
	\left( B\otimes I_x+\nabla F\left( \boldsymbol{u}^k \right) \right) \boldsymbol{u}^{k+1}=\boldsymbol{b}+\left( \nabla F\left( \boldsymbol{u}^k \right) \boldsymbol{u}^k-F\left( \boldsymbol{u}^k \right) \right),
	\label{eq2.4:energy}
\end{equation}
where $k\geqslant 0$ is the iteration index and $\nabla F\left( \boldsymbol{u}^k \right)$ =blk$\mathrm{diag}\left( \nabla f\left( {u_1}^k \right) ,\cdots ,\nabla f\left( {u_n}^k \right) \right)$ consists of the Jacobian matrix $\nabla f\left( {u_j}^k \right)$ as the $j$th block. In order to keep the diagonalization technique still available, we must replace or approximate all the blocks $\nabla f\left( {u_j}^k \right)$ with a single matrix $A_k$. Drawing on the ideas in \cite{2017Time}, we consider the following average Jacobian matrix 
\begin{equation*}
	A_k:=\frac{1}{n}\sum_{j=1}^n{\nabla f\left( {u_j}^k \right)}.
\end{equation*}
Therefore, a simple Kronecker product is presented to approximate $\nabla F\left( \boldsymbol{u}^k \right) $, that is, $\nabla F\left( \boldsymbol{u}^k \right) \approx I_t\otimes A_k$.
By substituting this into \eqref{eq2.4:energy}, we arrive at a simplified quasi-Newton iteration (SNI):
\begin{equation}
	\left( B\otimes I_x+I_t\otimes A_k \right) \boldsymbol{u}^{k+1}=\boldsymbol{b}+\left[ \left( I_t\otimes A_k \right) \boldsymbol{u}^k-F\left( \boldsymbol{u}^k \right) \right].
	\label{eq2.5:energy} 
\end{equation}
The convergence analysis of SNI can be found in \cite{PeterDeuflhard2006Newton,1970Iterative}.

Using the same structure of Jacobian system \eqref{eq2.5:energy} in SNI can also be solved in parallel. If $B$ is diagonalized into $B=VDV^{-1}$, then $\boldsymbol{u}^{k+1}$ in \eqref{eq2.5:energy} can be solved as 
\begin{equation}
	\begin{cases}  
		\boldsymbol{\tilde{g}} = (V^{-1} \otimes I_x) \boldsymbol{r}^k,   & \text{step-(a)}; \\  
		(\lambda_j I_x + A_k) z_j = \tilde{g}_j, \quad j = 1, 2, \ldots, n, &\text{step-(b)};\\  
		\boldsymbol{u}^{k+1} = (V \otimes I_x) \boldsymbol{z}, & \text{step-(c)};  
	\end{cases}
	\label{eq2.6:energy}
\end{equation}
where $\boldsymbol{r}^k=\boldsymbol{b}+\left( \left( I_t\otimes A_k \right) -F\left( \boldsymbol{u}^k \right) \right)$. In the linear case, $f\left( u(t) \right) =Au(t)$, we have $A_k=A$ and $\boldsymbol{r}^k=\boldsymbol{b}$. Then, \eqref{eq2.6:energy} reduces to \eqref{eq1.2:energy}. In \autoref{sec3}, we theoretically show that the matrix $B$ is indeed diagonalizable, and the expression of its decomposition is provided.

\subsection{Second-order problems}\label{sec2.2}
We consider the following second-order differential equation 
\begin{equation}
	u'' \left( t \right) +a_1u' \left( t \right) +f\left( u\left( t \right) \right) =0, \quad u\left( 0 \right) =u_0, u' \left( 0 \right) =\tilde{u}_0, t\in \left( 0,T \right],
	\label{eq2.7:energy}
\end{equation}
where $a_1\in \mathbb{R}$ or represent the matrix arising from some PDEs after a spatial discretization.
For discretization, we let $v\left( t \right) =u' \left( t \right)$ and make an order-reduction by rewriting \eqref{eq2.7:energy} as
\begin{equation}
	w' \left( t \right) :=\left[ \begin{array}{c}
		u\left( t \right)\\
		v\left( t \right)\\
	\end{array} \right] ^{\prime}=\left[ \begin{array}{c}
		v\left( t \right)\\
		-a_1v-f\left( u\left( t \right) \right)\\
	\end{array} \right] =:g\left( w \right) , \quad w\left( 0 \right) :=\left[ \begin{array}{c}
		u\left( 0 \right)\\
		v\left( 0 \right)\\
	\end{array} \right] =\left[ \begin{array}{c}
		u_0\\
		\tilde{u}_0\\
	\end{array} \right].
	\label{eq2.8:energy}
\end{equation}
In addition, by using the same time discretization scheme as in \autoref{sec2.1}, we have 
\begin{equation}
	\begin{cases}
		\frac{w_{j+1}-w_{j-1}}{2\varDelta t}+g\left( w_i \right) =0, \quad j=1,2,\cdots ,n-1,\\
		\frac{\frac{3}{2}w_n-2w_{n-1}+\frac{1}{2}w_{n-2}}{\varDelta t}+g\left( w_n \right) =0.
	\end{cases}
	\label{eq2.9:energy}
\end{equation}

Obviously, for \eqref{eq2.9:energy}, the form of the all-at-once system is the same as that of \eqref{eq2.3:energy}, so the diagonalization process \eqref{eq2.6:energy} can be directly applied. We eliminate the auxiliary variable $\boldsymbol{v}=\left( {v}_1^\top,\cdots ,{v}_n^\top \right) ^\top$, and  introduce the all-at-once system for  $\boldsymbol{u}$ in the next lemma. Thus, the equivalent system of \eqref{eq2.7:energy} reduces the storage requirement.
\begin{lemma}
	The numerical solution $\boldsymbol{u}=\left( {u}_1^\top,\cdots ,{u}_n^\top \right) ^\top$ of \eqref{eq2.9:energy} satisfies 
	\begin{equation}
		\left( B^2\otimes I_x+a_1 B\otimes I_x  \right) \boldsymbol{u}+F\left( \boldsymbol{u} \right) =\boldsymbol{b},
		\label{eq2.10:energy}
	\end{equation}
	where B is defined in \eqref{eq1.4:energy} and $\boldsymbol{b}=\left( \frac{a_1{u}_0^\top}{2\varDelta t}+\frac{{\tilde{u}_0}^\top}{2\varDelta t}, -\frac{{u}_0^\top}{4\varDelta t^2},0,\cdots ,0 \right) ^\top$.
	\begin{proof}
		Notice $w_j=\left( {u}_j^\top,{v}_j^\top \right) ^\top$, from \eqref{eq2.9:energy} we can represent $\left\{ u_j \right\}$ and $\left\{ v_j \right\}$ separately as 
		\begin{equation*}  
			\begin{cases}
				\frac{u_{j+1}-u_{j-1}}{2\varDelta t}-v_j=0 , \quad j=1,2,\cdots ,n-1,\\
				\frac{3u_n-4u_{n-1}+u_{n-2}}{2\varDelta t}-v_n=0,
			\end{cases}
		\end{equation*}
		and
		\begin{equation*}  
			\hspace{4em} \begin{cases}
				\frac{v_{j+1}-v_{j-1}}{2\varDelta t}+a_1v_j+f\left( u_j \right) =0, \quad j=1,2,\cdots ,n-1,\\
				\frac{3v_n-4v_{n-1}+v_{n-2}}{2\varDelta t}+a_1v_n+f\left( u_n \right) =0.
			\end{cases}
		\end{equation*}
		With the matrix $B$ given in \eqref{eq1.4:energy}, we have 
		\begin{equation}
			\left( B\otimes I_x \right) \boldsymbol{u}-\boldsymbol{v}=\boldsymbol{b}_1, \quad \left( B\otimes I_x \right) \boldsymbol{v}+a_1\boldsymbol{v}+F\left( \boldsymbol{u} \right) =\boldsymbol{b}_2,
			\label{eq2.11:energy}
		\end{equation}
		where $\boldsymbol{v}=\left( {v}_1^\top,\cdots ,{v}_n^\top \right)^\top$, $\boldsymbol{b}_1=\left( \frac{{u}_0^\top}{2\varDelta t},0,\cdots ,0 \right)^\top$ and $\boldsymbol{b}_2=\left( \frac{\tilde{u}_0^\top}{2\varDelta t},0,\cdots , 0 \right)^\top $. 
		Substituting the first equation of \eqref{eq2.11:energy} into the second one, gives $\left( B^2\otimes I_x+a_1B\otimes I_x \right) \boldsymbol{u}+F\left( \boldsymbol{u} \right) =\left( B\otimes I_x+a_1I \right) \boldsymbol{b}_1+\boldsymbol{b}_2$. After some simple calculations, \eqref{eq2.10:energy} can be obtained.
	\end{proof} 
\end{lemma}
\subsection{Third-order problems}\label{sec2.3}
In this part, we consider third-order differential equations:  
\begin{equation}
	u'''\left( t \right) +a_1u''\left( t \right) +a_2u'\left( t \right) +f\left( u\left( t \right) \right) =0, \quad u\left( 0 \right) =u_0, u' \left( 0 \right) =\tilde{u}_0, u'' \left( 0 \right) =\bar{u}_0, t\in \left( 0,T \right],
	\label{eq2.12:energy}
\end{equation} 
where $a_2 \in \mathbb{R}$ or represent the matrix arising from some PDEs after a spatial discretization.
Similar to the second-order case, let $v\left( t \right) =u' \left( t \right)$, we can reduce the order of \eqref{eq2.12:energy} as  
\begin{equation}
	w\left( t \right) =\left[ \begin{array}{c}
		u\left( t \right)\\
		v\left( t \right)\\
	\end{array} \right],
	\\
	P\left( w \right)=\left[ \begin{array}{c}
		-\bar{u}_0\left( 0 \right) -a_1\tilde{u}_0\left( 0 \right)\\
		a_2v\left( t \right) +f\left( u(t) \right)\\
	\end{array} \right], \quad
	w_0\left( 0 \right) =\left[ \begin{array}{c}
		u_0\\
		\tilde{u}_0\\
	\end{array} \right] =, \tilde{w}_0\left( 0 \right) =\left[ \begin{array}{c}
		\tilde{u}_0\\
		\bar{u}_0\\
	\end{array} \right].
	\label{eq2.13:energy}
\end{equation}
\begin{lemma}
	By using the same discretization as the second-order form, the vector $\boldsymbol{u}=\left( {u}_1^\top,\cdots ,{u}_n^\top \right)^\top$ satisfies 
	\begin{equation}
		\mathcal{A}\left( \mathcal{A}+a_2I \right) \boldsymbol{u}+\mathcal{A}\left( B^{-1}\otimes I_x \right) F\left( \boldsymbol{u} \right) =\left( \mathcal{A}+a_2I \right) \boldsymbol{b}_1+\mathcal{A}\left( B^{-1}\otimes I_x \right) \boldsymbol{b}_2,
		\label{eq2.14:energy}
	\end{equation}
	where $I=I_t\otimes I_x$, $\boldsymbol{b}_1=\left( \frac{a_1{u}_0^\top}{2\varDelta t}, -\frac{{u}_0^\top}{4\varDelta t^2}, 0,\cdots ,0 \right) ^\top$, $\boldsymbol{b}_2=\left( \frac{a_1{\tilde{u}_0}^\top+{\bar{u}_0}^\top}{2\varDelta t}, -\frac{{\tilde{u}_0}^\top}{4\varDelta t^2},0,\cdots ,0 \right) ^\top$ and $	\mathcal{A}=B^2\otimes I_x+a_1B\otimes I_x$.
\end{lemma}
\begin{proof}
	From \eqref{eq2.13:energy} and \eqref{eq2.12:energy}, we can obtain
	\begin{equation}
		\boldsymbol{w}'' +a_1\boldsymbol{w}' +P(\boldsymbol{w})=\boldsymbol{0},
		\label{eq2.15:energy}
	\end{equation}
	where $\boldsymbol{0}$ is a zero vector with suitable size, $\boldsymbol{w}=\left( {w}_1^\top,\cdots ,{w}_n^\top \right)^\top$ and $P(\boldsymbol{w})=\left( P(w_1)^\top,\cdots ,P(w_n)^\top \right)^\top$. It is a second-order form, and \eqref{eq2.15:energy} satisfies 
	\begin{equation}
		\left( B^2\otimes I_x+a_1B\otimes I_x \right) \boldsymbol{w}+P(\boldsymbol{w})=\boldsymbol{b},
		\label{eq2.16:energy}
	\end{equation}
	by matrix vectorization $vec(\cdot)$, we can obtain 
	\begin{equation}
		\left( B^2\otimes I_x+a_1B\otimes I_x \right) \boldsymbol{w}=vec\left( \hat{\boldsymbol{w}}\left( B^2+a_1B \right) ^\top \right),
		\label{eq2.17:energy}
	\end{equation}
	where  $\hat{\boldsymbol{w}}=\left( w_1\,\,w_2\,\,\cdots \,\,w_n \right) =\left( \begin{matrix}
		u_1&		u_2&		\cdots&		u_n\\
		v_1&		v_2&		\cdots&		v_n\\
	\end{matrix} \right) =\left( \begin{array}{c}
		\hat{\boldsymbol{u}}\\
		\hat{\boldsymbol{v}}\\
	\end{array} \right)  $, $vec\left( \hat{\boldsymbol{u}} \right) =\boldsymbol{u}$, $vec\left( \hat{\boldsymbol{v}} \right) =\boldsymbol{v}$, $\left( B^2+a_1B \right) ^\top=\left( d_1\,\,d_2\,\,\cdots \,\,d_n \right)$, and 
	\begin{equation*}
		vec\left( \hat{\boldsymbol{u}} \begin{pmatrix}
			d_1,&		\cdots&		d_n\\
		\end{pmatrix} \right) =\left( \begin{array}{c}
			\hat{\boldsymbol{u}}d_1\\
			\vdots\\
			\hat{\boldsymbol{u}}d_n\\
		\end{array} \right) , \quad vec\left( \hat{\boldsymbol{v}} \begin{pmatrix}
			d_1,&		\cdots&		d_n\\
		\end{pmatrix} \right) =\left( \begin{array}{c}
			\hat{\boldsymbol{v}}d_1\\
			\vdots\\
			\hat{\boldsymbol{v}}d_n\\
		\end{array} \right).
	\end{equation*}
	Substituting \eqref{eq2.17:energy} into \eqref{eq2.16:energy} gives 
	\begin{equation*}
		\begin{cases}
			\hat{\boldsymbol{u}}d_1-v' _1-a_1v_1=\frac{a_1u_0+\tilde{u}_0}{2\varDelta t},\\
			\hat{\boldsymbol{u}}d_2-v' _2-a_1v_2=-\frac{u_0}{4\varDelta t^2},\\
			\vdots\\
			\hat{\boldsymbol{u}}d_n-v' _n-a_1v_n=0,
		\end{cases}
	\end{equation*}
	and
	\begin{equation*}
		\begin{cases}
			\hat{\boldsymbol{v}}d_1+a_2v_1+f_1=\frac{a_1\tilde{u}_0+\bar{u}_0}{2\varDelta t},\\
			\hat{\boldsymbol{v}}d_2+a_2v_2+f_2=-\frac{\tilde{u}_0}{4\varDelta t^2},\\
			\vdots\\
			\hat{\boldsymbol{v}}d_n+a_2v_n+f_n=0.
		\end{cases}
	\end{equation*}
	Thus, with the matrix $B$ given by \eqref{eq1.4:energy} we have 
	\begin{equation}
		\left( B^2\otimes I_x+a_1B\otimes I_x \right) \boldsymbol{u}-\boldsymbol{v}' -a_1\boldsymbol{v}=\boldsymbol{b}_1, \quad \left( B^2\otimes I_x+a_1B\otimes I_x \right) \boldsymbol{v}+a_2\boldsymbol{v}+F\left( \boldsymbol{u} \right) =\boldsymbol{b}_2,
		\label{eq2.18:energy}
	\end{equation}
	where $\boldsymbol{v}=\left( {v}_1^\top,\cdots ,{v}_n^\top \right) ^\top, \boldsymbol{b}_1=\left( \frac{a_1{u}_0^\top+{\tilde{u}_0}^\top}{2\varDelta t}, -\frac{{u}_0^\top}{4\varDelta t^2}, 0\cdots ,0 \right) ^\top$, $\boldsymbol{b}_2=\left( \frac{a_1{\tilde{u}_0}^\top+{\bar{u}_0}^\top}{2\varDelta t}, -\frac{{\tilde{u}_0}^\top}{4\varDelta t^2}, 0,\cdots ,0 \right) ^\top$. Note that the dimension of the spatial identity matrix here should vary with the spatial dimensions of $\boldsymbol{u}$ and $\boldsymbol{w}$, that is, the $I_x$ corresponding to $\boldsymbol{u}$ is of $m$ dimensions, and the $I_x$ corresponding to $\boldsymbol{w}$ is of $2m$ dimensions. Substituting  
	\begin{equation*}
		\boldsymbol{v}' =\left( B\otimes I_x \right) \boldsymbol{v}-\left( \begin{array}{c}
			\frac{{v}_0^\top}{2\varDelta t}\\[0.5em]
			0\\
			\vdots\\
			0\\
		\end{array} \right),
	\end{equation*}
	into \eqref{eq2.18:energy} and eliminating $ \boldsymbol{v}$ yields  
	\begin{equation*}
		\mathcal{A}\left( \mathcal{A}+a_2I \right) \boldsymbol{u}+\mathcal{A}\left( B^{-1}\otimes I_x \right) F\left( \boldsymbol{u} \right) =\left( \mathcal{A}+a_2I \right) \boldsymbol{b}_1+\mathcal{A}\left( B^{-1}\otimes I_x \right) \boldsymbol{b}_2,
	\end{equation*}
	where $\mathcal{A}=B^2\otimes I_x+a_1B\otimes I_x$.
\end{proof}

If $f\left( u(t) \right) =Au(t)$, we have $F\left( \boldsymbol{u} \right) =\left( f^\top\left( u_1 \right) ,\cdots ,f^\top\left( u_n \right) \right) ^\top=\left( I_t\otimes A \right) \boldsymbol{u}$. Substituting it into  the all-at-once system \eqref{eq2.10:energy} and \eqref{eq2.14:energy}, and since  $B$  is diagonalizable, it is obvious that based on this connection, the previous direct PinT algorithm \eqref{eq2.6:energy} can be extended to \eqref{eq2.11:energy} and \eqref{eq2.18:energy}.
%\section{Discussion of \texorpdfstring{{\boldmath$Z=X \cup Y$}}{Z = X union Y}}
\section{The spectral decomposition of $B$}
\label{sec3}
In this section, we theoretically prove that the time-discrete matrix $B$ is diagonalizable and provide the explicit form of its spectral decomposition. Based on the analytical diagonalization, we estimate the condition number of the eigenvector matrix $V$ in the 2-norm, denoted as $Cond_2\left( V \right) =\mathcal{O} \left( n^3 \right) $. In addition, in the proof of the theorem, we used the explicit formula for the inverse tridiagonal Toeplitz matrix from \cite{2004On,1996Matrix}, and the recurrence relation from \cite{2009Introductory}.

To simplify the notation, we consider rescaling the diagonalization of matrix $B=\frac{1}{\varDelta t}\mathcal{B}$. Clearly, by diagonalizing $\mathcal{B} =V\varSigma{}V^{-1}$, we can obtain 
\begin{equation*}
	B=\frac{1}{\Delta t}\mathcal{B}=V(\frac{1}{\Delta t}\varSigma_{})V^{-1}=VDV^{-1}.
\end{equation*}
Define two functions 
\begin{equation*}
	T_n\left( x \right) =\cos \left( n\mathrm{arc}\cos x \right) , \quad U_n\left( x \right) =\sin \left[ \left( n+1 \right) \mathrm{arc}\cos x \right] /\sin \left( \mathrm{arc}\cos x \right) ,
\end{equation*}
which are the Chebyshev polynomials of the first and second kind, respectively. In the following theorem, we characterize the eigenvalues and eigenvectors of the time-discrete matrix $\mathcal{B}$ using Chebyshev polynomial representations. The imaginary unit is consistently denoted by $\mathrm{\mathbf{i}} = \sqrt{-1}$.

\begin{theorem}\label{thm3.1:pythagoras}
	The eigenvalues of $\mathcal{B}$ are $\lambda _j=\mathrm{\mathbf{i}}x_j$, with $\left\{ x_j \right\} _{j=1}^{n}$ being the $n$ roots of 
	\begin{equation}
		U_{n-1}\left( x \right) +\mathrm{\mathbf{i}}U_{n-2}\left( x \right) -\mathrm{\mathbf{i}}T_n\left( x \right) +T_{n-1}\left( x \right)=0.
		\label{eq3.1:energy}
	\end{equation}
	For each $\lambda _j$, the corresponding eigenvector is $P_j=\left[ p_{j,0},\cdots ,p_{j,n-1} \right] ^\top$, and 
	\begin{equation}
		p_{j,k}=\mathrm{\mathbf{i}}^kU_k\left( x_j \right) , \quad k=0,\cdots ,n-1,
		\label{eq3.2:energy}
	\end{equation}
	where $p_{j,0}=1$ is assumed for normalization.
	\begin{proof}
		Let $\lambda \in \mathbb{C}$ be an eigenvalue of matrix $\mathcal{B}$, and $P=\left[ p_0,p_1,\cdots ,p_{n-1} \right] ^\top\ne 0$ be the eigenvector corresponding to  $\lambda$. Thus, from $\mathcal{B} P=\lambda P$, we have 
		\begin{equation*}
			\begin{cases}
				\lambda p_0=\frac{p_1}{2},\\
				\lambda p_1=-\frac{p_0}{2}+\frac{p_2}{2},\\
				\vdots\\
				\lambda p_{n-2}=-\frac{p_{n-3}}{2}+\frac{p_{n-1}}{2},\\
				\lambda p_{n-1}=\frac{3}{2}p_{n-1}-2p_{n-2}+\frac{1}{2}p_{n-3}.\\
			\end{cases}
			\label{eq3.3:energy}
		\end{equation*}

		Clearly, $p_0\ne 0$, otherwise, $p_1=\cdots =p_{n-1}=0$. Without loss of generality, we normalize $p_0$, that is, assume $p_0=1$. Thus, we obtain $p_1=2\lambda$   and recursively derive 
		\begin{equation}
			2\lambda p_{k-1}=p_k-p_{k-2},
			\label{eq3.4:energy}
		\end{equation}
		holds for $k=2,\cdots ,n-1$. The last equation gives
		\begin{equation}
			\left( 2\lambda -3 \right) p_{n-1}=p_{n-3}-4p_{n-2}.
			\label{eq3.5:energy}
		\end{equation}

		Let $\lambda =\frac{1}{2}\left( y-\frac{1}{y} \right) =\mathrm{\mathbf{i}}\cos \theta$ with $y=\mathrm{\mathbf{i}}e^{\mathrm{\mathbf{i}}\theta}$. The general solution form of the recursive \eqref{eq3.4:energy} is
		\begin{equation}
			p_k=c_1y^k+c_2\left( -y \right) ^{-k}.
			\label{eq3.6:energy}
		\end{equation}
		From the initial conditions $p_0=1$  and  $p_1=2\lambda =y-y^{-1}$, it can be obtained that
		\begin{equation*}
			c_1+c_2=1, \quad c_1y-c_2y^{-1}=y-y^{-1},
		\end{equation*}
		which gives $c_1=\frac{y}{y+y^{-1}}$ and $c_2=\frac{y^{-1}}{y+y^{-1}}$. In addition, by substituting $y=\mathrm{\mathbf{i}}e^{\mathrm{\mathbf{i}}\theta}$, we get  
		\begin{equation}
			p_k=\frac{y^{k+1}+\left( -1 \right) ^ky^{-\left( k+1 \right)}}{y+y^{-1}}=\frac{\mathrm{\mathbf{i}}^k\sin \left[ \left( k+1 \right) \theta \right]}{\sin \theta}, \quad k=0,\cdots ,n-1.
			\label{eq3.7:energy}
		\end{equation}
		
		In view of $\lambda =\mathrm{\mathbf{i}}\cos \theta$, we rewrite \eqref{eq3.5:energy} as 
		\begin{equation*}
			\frac{\mathrm{\mathbf{i}}^{n-3}\sin \left[ \left( n-2 \right) \theta \right]}{\sin \theta}-4\frac{\mathrm{\mathbf{i}}^{n-2}\sin \left[ \left( n-1 \right) \theta \right]}{\sin \theta}=\left( 2\mathrm{\mathbf{i}}\cos \theta -3 \right) \frac{\mathrm{\mathbf{i}}^{n-1}\sin \left( n\theta \right)}{\sin \theta},
		\end{equation*}
		which is equivalent to 
		\begin{equation}
			\frac{\sin \left( n\theta \right)}{\sin \theta}+\mathrm{\mathbf{i}}\frac{\sin \left[ \left( n-1 \right) \theta \right]}{\sin \theta}=\mathrm{\mathbf{i}}\cos \left( n\theta \right) -\cos \left[ \left( n-1 \right) \theta \right] .
			\label{eq3.8:energy}
		\end{equation}

		This is a polynomial equation in $\lambda =\mathrm{\mathbf{i}}\cos \theta $ and has a dimension of $n$ due to 
		\begin{equation*}
			\frac{\sin \left( n\theta \right)}{\sin \theta}=U_{n-1}\left( -\mathrm{\mathbf{i}}\lambda \right), \quad \frac{\sin \left[ \left( n-1 \right) \theta \right]}{\sin \theta}=U_{n-2}\left( -\mathrm{\mathbf{i}}\lambda \right),
		\end{equation*}
		and
		\begin{equation*}
			\cos \left( n\theta \right) =T_n\left( -\mathrm{\mathbf{i}}\lambda \right), \quad \cos \left[ \left( n-1 \right) \theta \right] =T_{n-1}\left( -\mathrm{\mathbf{i}}\lambda \right),
		\end{equation*}
		which are polynomials of $\lambda$ with degrees $n-1$ and $n-2$ and $n$ and $n-1$, respectively. From \eqref{eq3.7:energy} and \eqref{eq3.8:energy}, the desired conclusion can be obtained.
	\end{proof}
\end{theorem}

Denote $\lambda =\mathrm{\mathbf{i}}x$ with $x=\cos \theta $. From \eqref{eq3.7:energy} and \eqref{eq3.8:energy}, we obtain 
\begin{equation}
	p_k=\mathrm{\mathbf{i}}^kU_k\left( x \right) , \quad k=0,1,\cdots ,n-1,
	\label{eq3.9:energy}
\end{equation}
and
\begin{equation}
	U_{n-1}\left( x \right) +\mathrm{\mathbf{i}}U_{n-2}\left( x \right)+T_{n-1}\left( x \right)-\mathrm{\mathbf{i}}T_n\left( x \right)=0.
	\label{eq3.10:energy}
\end{equation}
The $n$ roots $\{x_j\}_{j=1}^n$ of \eqref{eq3.10:energy} give the $n$ eigenvalues $\lambda _j=\mathrm{\mathbf{i}}x_j$ of the matrix $\mathcal{B}$. Eq.~\eqref{eq3.9:energy} gives the eigenvectors corresponding to each eigenvalue.
\begin{lemma}\label{lem3.1:example}
	There exist no common roots between the Chebyshev polynomials of the first kind $T_n(x)$ and those of the second kind $U_{n-1}(x)$.
\end{lemma}
\begin{proof}
	The roots of $T_n(x)$ and $U_{n-1}(x$) correspond to the $\cos \left( \frac{\left( 2k-1 \right) \pi}{2n} \right)$ ($k=1,2,\cdots ,n$) and $\cos \left( \frac{m\pi}{n} \right)$ ($m=1,2,\cdots ,n-1$), respectively. If $T_n(x)$ and $U_{n-1}(x$) have a common root, 
	then there exist $k_j$ and $m_i$ such that $\cos \left( \frac{\left( 2k_j-1 \right) \pi}{2n} \right)=\cos \left( \frac{m_i\pi}{n} \right)$, the implies $\frac{\left( 2k_j-1 \right) \pi}{2n}=\frac{m_i\pi}{n}$, we obtain $2k_j-1=2m_i$. Since the left-hand side is an odd integer while the right-hand side is an even integer, equality cannot hold. This contradiction implies that $T_n(x)$ and $U_{n-1}(x$) cannot any common roots.
\end{proof}

Based on \autoref{thm3.1:pythagoras} and \cref{lem3.1:example}, we further prove that the matrix  $\mathcal{B}$  has  $n$  distinct eigenvalues, that is,  $\mathcal{B}$  is diagonalizable and invertible. Subsequently, $B$ is nonsingular and diagonalizable.
\begin{theorem}\label{thm3.2:pythagoras}
	$U_{n-1}\left( x \right) +\mathrm{\mathbf{i}}U_{n-2}\left( x \right)+T_{n-1}\left( x \right)-\mathrm{\mathbf{i}}T_n\left( x \right)=0$ has n distinct roots, and if x is a root, then so is $-\bar{x}$.
	\begin{proof}
		For the convenience of the proof, we consider the  matrix $\mathcal{B}$ of order $n+1$ here. We first prove that the matrix  $\mathcal{B}$  has  $n+1$  distinct eigenvalues. Obviously, Eq.~\eqref{eq3.1:energy} has no real root (that is, $\lambda _j$ is not a pure imaginary number). Let the characteristic polynomial of matrix  $\mathcal{B}$  be  
		\begin{equation}
			f\left( \lambda \right) =\left| \lambda E-\mathcal{B} \right|=\left| \begin{matrix}
				D&		\alpha\\
				\beta&		a\\
			\end{matrix} \right|,
			\label{eq3.11:energy}
		\end{equation}
		where
		\begin{equation*}
			D=\left( \begin{matrix}
				\lambda&		-\frac{1}{2}&		&		\\
				\frac{1}{2}&		\lambda&		\ddots&		\\
				&		\ddots&		\ddots&		-\frac{1}{2}\\
				&		&		\frac{1}{2}&		\lambda\\
			\end{matrix} \right)\in \mathbb{R}^{n \times n}, \quad \alpha =\left( \begin{matrix}
				0,& \cdots,& 0,& -\frac{1}{2}\\
			\end{matrix} \right) ^\top\in \mathbb{R}^{n \times 1},
		\end{equation*} 
		and
		\begin{equation*}
			\beta =\left( \begin{matrix}
				0,&	\cdots,&	0,&	-\frac{1}{2},&  2\\
			\end{matrix} \right)\in \mathbb{R}^{1 \times n}, \quad a=\lambda -\frac{3}{2}.\hspace{10em}
		\end{equation*} 
		
		Notice that the matrix $D$ is invertible when $\lambda$ are not purely imaginary. We Define $$P=\left( \begin{matrix}
			I&		0\\
			-\beta D^{-1}&		1\\
		\end{matrix} \right) , \quad Q=\left( \begin{matrix}
			I&		-D^{-1}\alpha\\
			0&		1\\
		\end{matrix} \right) ,$$
		where $I\in \mathbb{R}^{n\times n}$ is the identity matrix. Thus, from 
		\begin{equation*}
			P\left( \begin{matrix}
				D&		\alpha\\
				\beta&		a\\
			\end{matrix} \right) Q=\left( \begin{matrix}
				D&		0\\
				0&		a-\beta D^{-1}\alpha\\
			\end{matrix} \right),
		\end{equation*}
		we arrive at 
		\begin{equation}
			\left| \begin{matrix}
				D&		\alpha\\
				\beta&		a\\
			\end{matrix} \right|=\left| \begin{matrix}
				D&		0\\
				0&		a-\beta D^{-1}\alpha\\
			\end{matrix} \right|=\left| D \right|\left| a-\beta D^{-1}\alpha \right|.
			\\
			\label{eq3.12:energy}
		\end{equation}

		Since $\lambda_j$ is not a pure imaginary number, \eqref{eq3.12:energy} holds if $\lambda=\lambda_j$. So $f\left( \lambda \right) =0$ is equivalent to 
		\begin{equation}
			a-\beta D^{-1}\alpha =0.
			\label{eq3.13:energy}
		\end{equation}

		Let $D^{-1}=\left( d_{ij} \right) (i,j=1,2,\cdots ,n)$, then $\beta D^{-1}\alpha =\frac{1}{4}d_{n-1,n}-d_{nn}$. Since the explicit formula for the elements of the inverse matrix of a Toeplitz tridiagonal matrix is 
		\begin{equation*}
			d_{ij}=\left( -1 \right) ^{i+j}\left( \frac{c}{b} \right) ^{\frac{j-i}{2}}\frac{U_{i-1}\left( \theta \right) U_{n-j}\left( \theta \right)}{\sqrt{bc}U_n\left( \theta \right)}, \quad   i\leqslant j, \theta =\frac{a}{2\sqrt{bc}}.
		\end{equation*}
		Here, $a$, $b$ and $c$ correspond to the elements of the main diagonal, upper diagonal, and subdiagonal respectively. It follows that 
		\begin{equation}
			d_{n-1,n}=-2\frac{U_{n-2}\left( \theta \right)}{U_n\left( \theta \right)}, \quad d_{nn}=-2i\frac{U_{n-1}\left( \theta \right)}{U_n\left( \theta \right)},
			\label{eq3.14:energy}
		\end{equation}
		where $\theta =\frac{a}{2\sqrt{bc}}=-\mathrm{\mathbf{i}}\lambda $. Substituting \eqref{eq3.14:energy} into \eqref{eq3.13:energy} and the recurrence relation of Chebyshev polynomials, we have
		\begin{align*}
			&\lambda -\frac{3}{2}=\frac{1}{4U_n\left( \theta \right)}\left( -2U_{n-2}\left( \theta \right) +8\mathrm{\mathbf{i}}U_{n-1}\left( \theta \right) \right), 
			\\
			&2U_n\left( \theta \right) \left( \lambda -\frac{3}{2} \right) =-U_{n-2}\left( \theta \right) +4\mathrm{\mathbf{i}}U_{n-1}\left( \theta \right), 
			\\
			&2U_n\left( \theta \right) \left( \lambda -\frac{3}{2} \right) =-\left( 2\theta U_{n-1}\left( \theta \right) -U_n\left( \theta \right) \right) +4\mathrm{\mathbf{i}}U_{n-1}\left( \theta \right). 
		\end{align*}
		Simplifying the above equation we get
		\begin{equation*}
			\left( 2\lambda -4 \right) U_n\left( \theta \right) =\left( 2\lambda +4 \right) \mathrm{\mathbf{i}}U_{n-1}\left( \theta \right).
			\label{eq3.15:energy}
		\end{equation*}

		From the relationship between Chebyshev polynomials of the first kind and the second kind, it can be transformed into  
		\begin{equation}
			\left( 2\lambda -4 \right) T_n\left( \theta \right) =\left( 2\lambda ^2-2\lambda +4 \right) \mathrm{\mathbf{i}}U_{n-1}\left( \theta \right).
			\label{eq3.16:energy}
		\end{equation}

		Taking the derivative of both sides of \eqref{eq3.16:energy} with respect to $\lambda$ yields 
		\begin{equation*}
			2T_n\left( \theta \right) -\left( 2\lambda -4 \right) nU_{n-1}\left( \theta \right) \mathrm{\mathbf{i}}=\left( 4\lambda -2 \right) \mathrm{\mathbf{i}}U_{n-1}\left( \theta \right) +\left( 2\lambda ^2-2\lambda +4 \right) \mathrm{\mathbf{i}}\frac{nT_n\left( \theta \right) -\theta U_{n-1}\left( \theta \right)}{\theta ^2-1}\left( -\mathrm{\mathbf{i}} \right).
		\end{equation*}
		Simplifying the above equation gives  
		\begin{equation}
			\begin{aligned}
				&\left( \lambda ^2+1 \right) \left[ 4\lambda -2+n\left( 2\lambda -4 \right) \right] \mathrm{\mathbf{i}}U_{n-1}\left( \theta \right) -2\left( \lambda ^2+1 \right) T_n\left( \theta \right)\\
				& -n\left( 2\lambda ^2-2\lambda +4 \right) T_n\left( \theta \right) -\mathrm{\mathbf{i}}\lambda \left( 2\lambda ^2-2\lambda +4 \right) U_{n-1}\left( \theta \right) =0.
			\end{aligned}
			\label{eq3.17:energy}
		\end{equation}
		
		We assume that \eqref{eq3.16:energy} has a multiple root  $\lambda_j$, and substitute \eqref{eq3.16:energy} into \eqref{eq3.17:energy} yields  
		%\begin{align*}
		%\lbrack &\left( \lambda _{j}^{2}+1 \right) \left( 2\lambda _j-4 \right) \left( 4\lambda _j-2+n\left( 2\lambda _j-4 \right) \right) -2\left( \lambda _{j}^{2}+1 \right) \left( 2\lambda %_{j}^{2}-2\lambda _j+4 \right) -n\left( 2\lambda _{j}^{2}-2\lambda _j+4 \right) ^2\\
		%&-\lambda _j\left( 2\lambda _j-4 \right) \left( 2\lambda _{j}^{2}-2\lambda _j+4 \right)  \rbrack U_{n-1}\left( \theta_j \right) =0,
		%\end{align*}
		%by simplifying the aforementioned equation, one obtains  
		\begin{equation}
			\left[ -\left( 4+8n \right) \lambda _{j}^{3}-12\lambda _{j}^{2} \right] U_{n-1}\left( \theta _j \right) =0.
			\label{eq3.18:energy}
		\end{equation}
		
		It is obvious that $-\left( 4+8n \right) \lambda _{j}^{3}-12\lambda _{j}^{2}$ has two roots as $0$ and one root as $-\frac{12}{4+8n}$, all of which are real roots. Thus, when $\lambda =\lambda _j$, $-\left( 4+8n \right) \lambda _{j}^{3}-12\lambda _{j}^{2}$ is not zero.
		For \eqref{eq3.18:energy} to hold, it can only be that  $U_{n-1}\left( \theta_j \right) =0$.
		Substituting this into \eqref{eq3.16:energy}, we have $T_n\left( \theta_j \right)=0$. This means that $\theta_j $  is a common root of  $T_n\left( \theta \right)$  and  $U_{n-1}\left( \theta \right)$. However, $T_n\left( \theta \right)$  and  $U_{n-1}\left( \theta \right)$   are Chebyshev polynomials of the first and second kinds respectively, and no common roots between them. Therefore, \eqref{eq3.11:energy} has  $n+1$  distinct roots, meaning the matrix  $\mathcal{B}$  has  $n+1$  distinct eigenvalues. Consequently, if $\mathcal{B}$ is of order $n$, there are $n$ distinct eigenvalues, so it is diagonalizable.
		
		By the properties of the first and second kind Chebyshev polynomials, we obtain  
		\begin{equation*}
			T_n\left( x \right) =\left( -1 \right) ^nT_n\left( -x \right) , \quad U_{n-1}\left( x \right) =\left( -1 \right) ^{n-1}U_{n-1}\left( -x \right).
		\end{equation*}
		If $x$ is a root, then 
		\begin{align*}
			U_{n-1}\left( -\bar{x} \right) +\mathrm{\mathbf{i}}U_{n-2}\left( -\bar{x} \right) &=\left( -1 \right) ^{n-1}\bar{U}_{n-1}\left( x \right) +\left( -1 \right) ^{n-1}\mathrm{\mathbf{i}}\bar{U}_{n-2}\left( x \right)\\ &=\left( -1 \right) ^{n-1}\left( \bar{U}_{n-1}\left( x \right) +\mathrm{\mathbf{i}}\bar{U}_{n-2}\left( x \right) \right)\\
			& =\left( -1 \right) ^{n-1}\left( -\mathrm{\mathbf{i}}\bar{T}_n\left( x \right) -\bar{T}_{n-1}\left( x \right) \right)\\
			& =\left( -1 \right) ^n\mathrm{\mathbf{i}}\bar{T}_n\left( x \right) -\left( -1 \right) ^{n-1}\bar{T}_{n-1}\left( x \right) \\
			&=\mathrm{\mathbf{i}}T_n\left( -\bar{x} \right) -T_{n-1}\left( -\bar{x} \right) .
		\end{align*}
		This implies that $-\bar{x}$ is also a root.
	\end{proof}
\end{theorem}

According to \autoref{thm3.2:pythagoras}, we conclude that $\mathcal{B}$ is diagonalizable. Let the diagonalization of $\mathcal{B}$ be $\mathcal{B} =V\varSigma{V}$, where $\varSigma{=\mathrm{diag}\left( \lambda _1,\cdots ,\lambda _n \right)}$ and 
\begin{equation}
	V=\left[ P_1,P_2,\cdots ,P_n \right] =\underset{:=E}{\underbrace{\mathrm{diag}\left( \mathrm{\mathbf{i}}^0,\mathrm{\mathbf{i}}^1,\cdots ,\mathrm{\mathbf{i}}^{n-1} \right) }}\underset{:=\varPhi}{\underbrace{ \begin{bmatrix}
				U_0\left( x_1 \right)  &  \cdots  &  U_0\left( x_n \right)\\
				\vdots  &  \cdots  &  \vdots\\
				U_{n-1}\left( x_1 \right)  &  \cdots  &  U_{n-1}\left( x_n \right)\\
	\end{bmatrix}}}=E\varPhi,
	\label{eq3.19:energy}
\end{equation}
where $\left\{ \lambda _j \right\} _{j=1}^{n}$ and $\left\{ x_j \right\} _{j=0}^{n-1}$ are given by \autoref{thm3.1:pythagoras}. In \eqref{eq3.19:energy}, $E$ is diagonal matrix and $\varPhi $ is a Vandermonde matrix \cite{2012Differential} defined by Chebyshev orthogonal polynomials. Thus, we can obtain 
\begin{equation}
	Cond_2\left( V \right) =Cond_2\left( I\varPhi \right) =Cond_2\left( \varPhi \right).
	\label{eq3.20:energy}
\end{equation}

The following theorem shows that $Cond_2\left( V \right) =\mathcal{O} \left( n^3 \right) $, which means that as $n$ increases, the rounding errors generated by the diagonalization process only increase moderately.
%\begin{lemma}
%Let $Q=\left\{ q_{ij} \right\} $ be an $n\times n$ matrix, we have $\mathrm{tr}\left( Q \right) =\sum_{i=1}^n{\lambda _i}$ and $\max \left\{ \lambda _1,\lambda _2,\cdots ,\lambda _n %\right\} \leqslant \sum_{i=1}^n{\left| q_{ii} \right|}$, where $\lambda_i$ is an eigenvalue of Q and $j,i=1,2,\cdots ,n$.
%\begin{proof}
%Let the characteristic polynomial of $Q$ be $f=\lambda ^n+a_{n-1}\lambda ^{n-1}+\cdots +a_1\lambda +a_0$. Then, according to Vieta's formulas\cite{sturmfels2002solving}, we can obtain %$a_{n-1}=-\left( \lambda _1+\lambda _2+\cdots +\lambda _n \right)$. Since $\mathrm{tr}\left( Q \right) =\sum_{i=1}^n{q_{ii}}$, and the characteristic polynomial of $Q$ is 
%\begin{equation*}
%\left| \lambda I-Q \right|=\begin{vmatrix}
	%	\lambda -q_{11}&		\cdots&		-q_{1n}\\
	%	\,\,\vdots&		\cdots&		\,\,\vdots\\
	%	-q_{n1}&		\cdots&		\lambda -q_{nn}\\
	%\end{vmatrix}.
	%\end{equation*} 
	%By expanding it, we can see that the coefficient of $\lambda ^{n-1}$ is $a_{n-1}=-\sum_{i=1}^n{q_{ii}}=-\mathrm{tr}\left( Q \right)$. Thus,  $\mathrm{tr}\left( Q \right) %=\sum_{i=1}^n{\lambda _i}$.
	%\end{proof}
	%\end{lemma}
	\begin{theorem}
		For the eigenvector matrix $V$ defined in \eqref{eq3.9:energy}, it holds
		\begin{equation}
			Cond_2\left( V \right) =\mathcal{O} \left( n^3 \right).
			\label{eq3.21:energy} 
		\end{equation}
		\begin{proof}
			\textbf{Step 1}: estimate $\left\| \varPhi \right\| _2$. Using the Christoffel-Darboux formula\cite{szego1975orthogonal}, we arrive at  
			\begin{equation}
				\left( \varPhi ^*\varPhi \right) _{jk}=\sum_{l=0}^{n-1}{U_l\left( \bar{x}_j \right)}U_l\left( x_k \right) =\frac{U_n\left( \bar{x}_j \right) U_{n-1}\left( x_k \right) -U_{n-1}\left( \bar{x}_j \right) U_n\left( x_k \right)}{2\left( \bar{x}_j-x_k \right)}.
				\label{eq3.23:energy}
			\end{equation}

			Let $\theta _j=\alpha _j+\mathrm{\mathbf{i}}\beta _j$, so $x_j=\cos \theta _j$. From the initial value of the iteration in \autoref{sec4.1}, we have $\alpha_j =\mathcal{O} \left( n^{-1} \right)$ and $\beta_j =\mathcal{O} \left( n^{-1} \right)$ ($j=1, \cdots, n$). We first estimate $\left| \sin \bar{\theta}_j\sin \theta _k \right|$ and $\left| \bar{x}_j-x_k \right|$. 
			\begin{equation*}
				\begin{aligned}
					\left| \sin \bar{\theta}_j \right|&=\left| \sin \left( \alpha _j-\mathrm{\mathbf{i}}\beta _j \right) \right|=\left| \sin \alpha _j\cos \mathrm{\mathbf{i}}\beta _j-\cos \alpha _j\sin \mathrm{\mathbf{i}}\beta _j \right|=\left| \sin \alpha _j\cosh \beta _j-\mathrm{\mathbf{i}}\cos \alpha _j\sinh \beta _j \right|\\
					&=\sqrt{\sin ^2\alpha _j\cosh ^2\beta _j+\cos ^2\alpha _j\sinh ^2\beta _j}\geqslant \left| \sinh \beta _j \right|\geqslant \left| \beta _j \right|=\mathcal{O} \left( n^{-1} \right).
				\end{aligned} 
			\end{equation*}
			Analogously, $\left| \sin \theta _k \right|\geqslant \mathcal{O} \left( n^{-1} \right)$. Thus $\left| \sin \bar{\theta}_j\sin \theta _k \right|\geqslant \mathcal{O} \left( n^{-2} \right)$.
			Now, we turn to estimate $\left| \bar{x}_j-x_k \right|$. We first have
			\begin{equation*}
				\begin{aligned}
					\sin \left( \frac{\bar{\theta}_j+\theta _k}{2} \right) &=\sin \left( \frac{\alpha _j-\mathrm{\mathbf{i}}\beta _j+\alpha _k+\mathrm{\mathbf{i}}\beta _k}{2} \right) =\sin \left( \frac{\alpha _j+\alpha _k+\mathrm{\mathbf{i}}\left( \beta _k-\beta _j \right)}{2} \right) \\
					&=\sin \left( \frac{\alpha _j+\alpha _k}{2} \right) \cos \left( \frac{\mathrm{\mathbf{i}}\left( \beta _k-\beta _j \right)}{2} \right) +\cos \left( \frac{\alpha _j+\alpha _k}{2} \right) \sin \left( \frac{\mathrm{\mathbf{i}}\left( \beta _k-\beta _j \right)}{2} \right)\\
					& =\sin \left( \frac{\alpha _j+\alpha _k}{2} \right) \cosh \left( \frac{\beta _k-\beta _j}{2} \right) +\mathrm{\mathbf{i}}\cos \left( \frac{\alpha _j+\alpha _k}{2} \right) \sinh \left( \frac{\beta _k-\beta _j}{2} \right).
				\end{aligned}
			\end{equation*}
			This implies that 
			\begin{equation*}
				\begin{aligned}
					\left| \sin \left( \frac{\bar{\theta}_j+\theta _k}{2} \right) \right|&=\sqrt{\sin ^2\left( \frac{\alpha _j+\alpha _k}{2} \right) \cosh ^2\left( \frac{\beta _k-\beta _j}{2} \right) +\cos ^2\left( \frac{\alpha _j+\alpha _k}{2} \right) \sinh ^2\left( \frac{\beta _k-\beta _j}{2} \right)}\\
					&\geqslant \left| \sinh \left( \frac{\beta _k-\beta _j}{2} \right) \right|\geqslant \left| \frac{\beta _k-\beta _j}{2} \right|=\mathcal{O} \left( n^{-1} \right).
				\end{aligned}
			\end{equation*}
			Similarly, $\left| \sin \left( \frac{\bar{\theta}_j-\theta _k}{2} \right) \right|\geqslant \mathcal{O} \left( n^{-1} \right)$. Therefore,
			\begin{equation*}
				\left| \bar{x}_j-x_k \right|=\left| \cos \bar{\theta}_j-\cos \theta _k \right|=2\left| \sin \left( \frac{\bar{\theta}_j+\theta _k}{2} \right) \sin \left( \frac{\bar{\theta}_j-\theta _k}{2} \right) \right|\geqslant \mathcal{O} \left( n^{-2} \right).
			\end{equation*}
			We proceed to estimate $\left| \sin \left[ \left( n+1 \right) \bar{\theta}_j \right] \sin \left( n\theta _k \right) -\sin \left[ \left( n+1 \right) \theta _k \right] \sin \left( n\bar{\theta}_j \right) \right|$. Since $n\beta _j=\mathcal{O} \left( 1 \right)$,
			\begin{equation*}
				\begin{aligned}
					\left| \sin \left( n\bar{\theta}_j \right) \right|&=\left| \sin \left( n\left( \alpha _j-\mathrm{\mathbf{i}}\beta _j \right) \right) \right|=\left| \sin n\alpha _j\cosh n\beta _j-\mathrm{\mathbf{i}}\cos n\alpha _j\sinh n\beta _j \right|\\
					&=\sqrt{\sin ^2n\alpha _j\cosh ^2n\beta _j+\cos ^2n\alpha _j\sinh ^2n\beta _j}\leqslant \left| \cosh n\beta _j \right|,
				\end{aligned}
			\end{equation*}
			this implies that $\left| \cosh n\beta _j \right|$ is bounded. Thus, $\left| \sin \left[ \left( n+1 \right) \bar{\theta}_j \right] \sin \left( n\theta _k \right) -\sin \left[ \left( n+1 \right) \theta _k \right] \sin \left( n\bar{\theta}_j \right) \right|$ is bounded, we can define $\left| \sin \left[ \left( n+1 \right) \bar{\theta}_j \right] \sin \left( n\theta _k \right) -\sin \left[ \left( n+1 \right) \theta _k \right] \sin \left( n\bar{\theta}_j \right) \right|\leqslant C$ ($C$ is a constant).   From \eqref{eq3.23:energy} for any $j,k=1,2,\cdots ,n$, we get
			\begin{align*}
				\left| \left( \varPhi ^*\varPhi \right) _{jk} \right|&=\left| \sum_{l=0}^{n-1}{U_l\left( \bar{x}_j \right)}U_l\left( x_k \right) \right|=\left| \frac{U_n\left( \bar{x}_j \right) U_{n-1}\left( x_k \right) -U_{n-1}\left( \bar{x}_j \right) U_n\left( x_k \right)}{2\left( \bar{x}_j-x_k \right)} \right|\\
				&=\left| \frac{\sin \left[ \left( n+1 \right) \bar{\theta}_j \right] \sin \left( n\theta _k \right) -\sin \left[ \left( n+1 \right) \theta _k \right] \sin \left( n\bar{\theta}_j \right)}{2\sin \bar{\theta}_j\sin \theta_k \left( \bar{x}_j-x_k \right)} \right|\\
				&\leqslant \frac{C}{2\mathcal{O} \left( n^{-2} \right) \mathcal{O} \left( n^{-2} \right)}=\mathcal{O} \left( n^4 \right).  
			\end{align*}
			
			Then $\sum_{j=1}^n{\left| \left( \varPhi ^*\varPhi \right) _{jk} \right|}\leqslant \mathcal{O} \left( n^5 \right)$ and 
			\begin{equation}
				\left\| \varPhi \right\| _2=\sqrt{\rho \left( \varPhi ^*\varPhi \right)}\leqslant \sqrt{\left\| \varPhi ^*\varPhi \right\| _1}\leqslant \mathcal{O} \left( n^{\frac{5}{2}} \right).
				\label{eq3.24:energy}
			\end{equation}
			\textbf{Step 2}: estimate $\left\| \varPhi^{-1} \right\| _2$. Let
			\begin{equation}
				\begin{aligned}
					L_j\left( x \right) &=\prod_{1\leqslant k\leqslant n,k\ne j}^{}{\frac{x-x_k}{x_j-x_k}}\\
					&=\frac{U_{n-1}\left( x \right) +\mathrm{\mathbf{i}}U_{n-2}\left( x \right) -\mathrm{\mathbf{i}}T_n\left( x \right) +T_{n-1}\left( x \right)}{\left( x-x_j \right) \left[ {U'}_{n-1}\left( x \right) +{\mathrm{\mathbf{i}}U'}_{n-2}\left( x \right) -{\mathrm{\mathbf{i}}T'}_n\left( x \right) +{T'}_{n-1}\left( x \right) \right]}, \quad j=1, 2, \cdots, n,
					\label{eq3.22:energy}
				\end{aligned}
			\end{equation}
			the Lagrange interpolation polynomial $L_j\left( x_k \right) =\delta _{jk}$ and both $x_k$, $x_j$ are roots of \eqref{eq3.1:energy}, where $\delta _{jk}$ is the standard Kronecker delta.
			We denote $W={\left( w_{jk} \right)} ^{n}_{j,k=1}=\varPhi ^{-1}$, and obtain from orthogonality and Gaussian quadeature formula that 
			\begin{equation}
				w_{jk}=\frac{2}{\pi}\int_{-1}^1{L_j\left( x \right) U_{k-1}\left( x \right) \sqrt{1-x^2}dx}=\sum_{s=1}^n{\frac{2\left( 1-{y_s}^2 \right)}{n+1}L_j}\left( y_s \right) U_{k-1}\left( y_s \right),
				\label{eq3.25:energy}
			\end{equation}
			where $y_s=\cos \theta _s$ with $\theta _s=\frac{s\pi}{n+1}$ ($s=1,\cdots ,n$). 
			A simple calculation yields 
			\begin{equation*}
				\left( WW^* \right) _{jk}=\sum_{s=1}^n{\frac{2\left( 1-{y_s}^2 \right)}{n+1}}L_j\left( y_s \right) \bar{L}_k\left( y_s \right).
			\end{equation*}
			We denote
			\begin{align*}
				&\left| L_p\left( y_s \right) \bar{L}_q\left( y_s \right) \right|=\max \left\{ \left| L_j\left( y_s \right) \bar{L}_k\left( y_s \right) \right| \right\} (j,k=1,\cdots,n),\\
				&\left| L_p\left( y_m \right) \bar{L}_q\left( y_m \right) \right|=\max \left\{ \left| L_p\left( y_1 \right) \bar{L}_q\left( y_1 \right) \right|, \cdots , \left| L_p\left( y_n \right) \bar{L}_q\left( y_n \right) \right| \right\}.
			\end{align*} 
			Thus, for $k=1, \cdots, n$, we conclude that
			\begin{align*}
				\sum_{j=1}^n{\left| \left( WW^* \right) _{jk} \right|}&= \sum_{j=1}^n \left|\sum_{s=1}^n{\frac{2\left( 1-{y_s}^2 \right)}{n+1}}L_j\left( y_s \right) \bar{L}_k\left( y_s \right) \right|\leqslant \frac{2}{n+1}\sum_{s=1}^n{\left( 1-y_s \right) \sum_{j=1}^n{\left| L_j\left( y_s \right) \bar{L}_k\left( y_s \right) \right|}}\\
				&\leqslant \frac{2n}{n+1} \sum_{s=1}^n{\left( 1-y_s \right) \left| L_p\left( y_s \right) \bar{L}_q\left( y_s \right) \right|} \leqslant \frac{2n}{n+1}\sum_{s=1}^n{\left| L_p\left( y_s \right) \bar{L}_q\left( y_s \right) \right|}\\
				&\leqslant \frac{2n^2}{n+1}\left| L_p\left( y_m \right) \bar{L}_q\left( y_m \right) \right|=\mathcal{O} \left( n \right),
			\end{align*}
			and
			\begin{equation}
				\left\| W \right\| _2=\sqrt{\rho \left( WW^* \right)}\leqslant \sqrt{\left| W^*W \right|_1}=\mathcal{O} \left( n^{\frac{1}{2}} \right). 
				\label{eq3.26:energy}
			\end{equation}
			\textbf{Step 3}: estimate $Cond_2\left( V \right)$. Combining \eqref{eq3.24:energy} and \eqref{eq3.26:energy}, gives $Cond_2\left( V \right) =Cond_2\left( \varPhi \right) =\left\| \varPhi \right\| _2\left\| \varPhi ^{-1} \right\| _2=\left\| \varPhi \right\| _2\left\| W \right\| _2=\mathcal{O} \left( n^3 \right)$.
		\end{proof}
	\end{theorem}
	
	\section{Fast implementation}
	\label{sec4}
	We observe that the acceleration of our algorithm \eqref{eq2.6:energy} can be achieved by fast computation of the spectral decomposition of $B$ and the inversion of $V^{-1}$. In this section, we present a fast algorithm for  $V^{-1}$  and a fast spectral decomposition for  $B$.
	
	\subsection{Fast spectral decomposition of $B$}
	\label{sec4.1}
	The spectral decomposition of the time-discrete matrix  $B=VDV^{-1}$  is crucial in our PinT algorithm \eqref{eq2.6:energy}. The eigenvalues  $\lambda_j$  can be calculated via Newton's iterative method. According to \autoref{thm3.1:pythagoras}, we have  $\lambda _j=\mathrm{\mathbf{i}}\cos \left( \theta _j \right)$, where $\theta _j$ is the $j$th root of 
	\begin{equation*}
		\rho \left( \theta \right) :=\sin \left( n\theta \right) +\mathrm{\mathbf{i}}\sin \left( \left( n-1 \right) \theta \right) +\sin \left( \theta \right) \cos \left( \left( n-1 \right) \theta \right) -\mathrm{\mathbf{i}}\sin \left( \theta \right) \cos \left( n\theta \right) =0.
	\end{equation*}
	Applying Newton’s iteration to this nonlinear equation, we have
	\begin{equation}
		\theta _{j}^{\left( l+1 \right)}=\theta _{j}^{\left( l \right)}-\frac{\rho \left( \theta _{j}^{\left( l \right)} \right)}{\rho' \left( \theta _{j}^{\left( l \right)} \right)}, \quad l=0,1,2,\cdots .
		\label{eq4.1:energy}
	\end{equation}
	
	The Newton iteration executes $n$ cycles corresponding to the $n$ eigenvalues $\lambda_j$. The number of iterations for all  $n$  eigenvalues remains almost constant, so the complexity of calculating all eigenvalues via the Newton iteration in \eqref{eq4.1:energy} is  $\mathcal{O} \left( n \right)$. However, selecting  $n$  initial guesses  $\left\{ \theta _{j}^{\left( l \right)} \right\} _{j=1}^{n}$  is very difficult, as  $n$  iterations must converge to  $n$  distinct values—meaning not all eigenvalues can be found. Here, we choose the iterative initial values as  
	\begin{equation*}
		\theta _{j}^{\left( 0 \right)}=j\pi \left( \frac{c}{n}+\frac{1-c}{n+1} \right) +\mathrm{\mathbf{i}}\left( \frac{1-c}{n-2}+\frac{c}{n-1} \right), \quad c=0.7, j=1,2,\cdots ,n.
	\end{equation*}
	This ensures that the iteration \eqref{eq4.1:energy} converges to the $n$ different eigenvalues correctly.
	\subsection{ A fast $\mathcal{O} \left( n^2 \right) $ algorithm for computing $V^{-1}$}
	According to \eqref{eq3.19:energy}, the eigenvector matrix $V=E\varPhi$. To compute  $V^{-1}=\varPhi^{-1}E^{-1}$  is essentially to determine  $W=\varPhi^{-1}$. In this part, we propose a fast method for calculating  $V^{-1}$.
	
	Applying the recurrence relation of the second kind Chebyshev polynomials $2yU_j\left( y \right) =U_{j+1}\left( y \right) +U_{j-1}\left( y \right)$ gives 
	\begin{equation}
		\begin{aligned}
			4{y_s}^2U_{k-1}\left( y_s \right) &=2y_s\left[ U_{k-2}\left( y_s \right) +U_k\left( y_s \right) \right]\\
			&= \begin{cases}
				U_{k-3}\left( y_s \right) +2U_{k-1}\left( y_s \right) +U_{k+1}\left( y_s \right) , \quad 2\leqslant k\leqslant n-1,\\
				U_{k-1}\left( y_s \right) +U_{k+1}\left( y_s \right) , \quad  k=1,\\
				U_{k-3}\left( y_s \right) +U_{k-1}\left( y_s \right) , \quad k=n.\\
			\end{cases}
		\end{aligned}
		\label{eq4.2:energy}
	\end{equation}
	Then, it follows from \eqref{eq3.25:energy} that
	\begin{equation}
		\begin{aligned}
			2w_{jk}&=\frac{1}{n+1}\sum_{s=1}^n 4L_j\left( y_s \right) U_{k-1}\left( y_s \right) -\frac{1}{n+1} \sum_{s=1}^n 4y_{s}^{2}L_j\left( y_s \right) U_{k-1}\left( y_s \right) \\
			&= \begin{cases}
				2\psi _{j,k}-\psi _{j,k-2}-\psi _{j,k+2}, \quad 2\leqslant k\leqslant n-1,\\
				3\psi _{j,k}-\psi _{j,k-2},  \quad k=1,\\
				3\psi _{j,k}-\psi _{j,k+2}, \quad  k=n,\\
			\end{cases}
		\end{aligned}
		\label{eq4.3:energy}
	\end{equation}
	where $\psi _{j,k}=\frac{1}{n}\sum_{s=1}^n{L_j\left( y_s \right) U_{k-1}\left( y_s \right)}$. Since $U_n\left( y_s \right) =0$, we have $\psi _{j,n+1}=0$ for $j=1,2,\cdots ,n.$ Define 
	\begin{equation}
		p_n\left( x \right) =U_{n-1}\left( x \right) +T_{n-1}\left( x \right) +\mathrm{\mathbf{i}}U_{n-2}\left( x \right) -\mathrm{\mathbf{i}}T_n\left( x \right), \quad
		b_k=\frac{1}{n+1}\sum_{s=1}^n{p_n\left( y_s \right) U_{k-1}\left( y_s \right)}.
		\label{eq4.4:energy}
	\end{equation}

	From \eqref{eq3.22:energy}, we know that $p_n\left( x \right) =p_{n}^{\prime}\left( x \right) \left( y_s-x_j \right) L_j\left( y_s \right)$. Thus,

	\begin{equation}
		\frac{2b_k}{p_{n}^{\prime}\left( x_j \right)}=\frac{1}{n+1}\sum_{s=1}^n{2\left( y_s-x_j \right) L_j\left( y_s \right) U_{k-1}\left( y_s \right) =\psi _{j,k-1}+\psi _{j,k+1}-2x_j\psi _{j,k}}.
		\label{eq4.5:energy}
	\end{equation}

	To evaluate $b_k$, we investigate the integral of $p_n\left( x \right) U_{k-1}\left( x \right) \sqrt{1-x^2}$ on $[-1,1]$. According to \eqref{eq4.2:energy} and \eqref{eq4.4:energy}, we obtain by the Gaussian quadrature formula that 
	\begin{equation}
		\begin{aligned}
			\frac{4}{\pi}\int_{-1}^1 p_n\left( x \right) U_{k-1}\left( x \right) \sqrt{1-x^2}dx
			&=\frac{1}{n+1}\sum_{s=1}^n 4\left( 1-y_{s}^{2} \right) p_n\left( y_s \right) U_{k-1}\left( y_s \right)\\
			&= \begin{cases}
				2b_k-b_{k-2}-b_{k+2}, \quad 2\leqslant k\leqslant n-1,\\
				3b_k-b_{k+2}, \quad k=1,\\
				3b_k-b_{k-2}, \quad k=n.\\
			\end{cases}
		\end{aligned}
		\label{eq4.6:energy}
	\end{equation}

	Additionally, it can be calculated according to the orthogonality of Chebyshev polynomials that 
	\begin{equation}
		\frac{4}{\pi}\int_{-1}^1{p_n\left( x \right) U_{k-1}\left( x \right) \sqrt{1-x^2}dx}=2\delta _{n-1,k-1}+\delta _{n-1,k-1}-\delta _{n-3,k-1}+3\mathrm{\mathbf{i}}\delta _{n-2,k-1}-\mathrm{\mathbf{i}}\delta _{n,k-1}.
		\label{eq4.7:energy}
	\end{equation}
	By combining Eqs. \eqref{eq4.6:energy} and \eqref{eq4.7:energy}, a linear system is obtained 
	\begin{equation}
		S_n\boldsymbol{b}:=\begin{pmatrix}
			3&		0&		-1&		&		&		&		\\
			0&		2&		0&		-1&		&		&		\\
			-1&		0&		2&		0&		-1&		&		\\
			&		\ddots&		\ddots&		\ddots&		\ddots&		\ddots&		\\
			&		&		-1&		0&		2&		0&		-1\\
			&		&		&		-1&		0&		2&		0\\
			&		&		&		&		-1&		0&		3\\
		\end{pmatrix} \left( \begin{array}{c}
			b_1\\
			b_2\\
			b_3\\
			\vdots\\
			b_{n-2}\\
			b_{n-1}\\
			b_n\\
		\end{array} \right) =\left( \begin{array}{c}
			0\\
			0\\
			0\\
			\vdots\\
			-1\\
			3\mathrm{\mathbf{i}}\\
			3\\
		\end{array} \right).
		\label{eq4.8:energy}
	\end{equation}
	
	Let $\varPsi =\left[ \psi _{jk} \right]$. A fast algorithm for computing  $W=\varPhi ^{-1}$  is given by the following three steps.
	
	Step-1: solve $\boldsymbol{b}=\left[ b_1,b_2,\cdots ,b_n \right] ^\top$  from \eqref{eq4.8:energy}, and the computational complexity is  $\mathcal{O} \left( n \right) $.
	
	Step-2: Based on the fact $\psi _{j,n+1}=0$, the $j$th row $\boldsymbol{\varPsi }_j=\left[ \psi _{j,1},\varPsi _{j,2},\cdots ,\psi _{j,n} \right] $ of $\varPsi$ can be solved from a sequence of sparse tridiagonal linear systems 
	\begin{equation}
		G_j\psi _{j}^{T}:=\text{Tri}\mathrm{diag}\left\{ 1,-2x_j,1 \right\} \boldsymbol{\varPsi }_{j}^{T}=\frac{2}{p_{n}^{\prime}\left( x_j \right)}\boldsymbol{b}, \quad j=1,2, \cdots, n.
		\label{eq4.9:energy}
	\end{equation}
	Based on the fast Thomas algorithm applied to each system, the computational complexity is $\mathcal{O} \left( n^2 \right)$.

	Step-3: Compute $W=\frac{1}{2}\varPsi S_n$, which needs $\mathcal{O} \left( n^2 \right)$ operations since $S_n$ is a sparse matrix.
	
	Thus, the matrix  $V^{-1}$  can be computed with the complexity of  $\mathcal{O} \left( n^2 \right)$.
	\section{Numerical experiments}
	\label{sec5}
	In this section, we present several  numerical experiments to validate the acceleration efficiency of the proposed direct PinT algorithm \eqref{eq2.6:energy}. We also provide a comparison of CPU time between our fast spectral decomposition of $B$ in \autoref{sec4} and the MATLAB built-in function $\texttt{eig}$. 
	All experiments are conducted using MATLAB2022a on a Intel(R) Core(TM) i7-14700K  3.40GHz of processor and 32GB of RAM.
	\subsection{Fast spectral decomposition of $B$}
	\label{sec5.1}
	In \autoref{sec4}, we introduce the fast spectral decomposition of $B$  and the fast solution for  $V^{-1}$, improving the parallel efficiency of step-(b) by explicitly constructing the matrix  $V^{-1}$.
	Let $B=V_{eig}D_{eig}V_{eig}^{-1}$ and $B=V_{fast}D_{fast}V_{fast}^{-1}$ are the spectral decompositions corresponding to the \texttt{eig} and our fast algorithm, respectively. Define the maximum relative difference as  
	\begin{table}[h!]
		\centering
		\caption{Comparison of \texttt{eig+mrdivide} solver and our fast spectral decomposition algorithm.}
		\label{tab:comparison}
		\begin{tabular}{ccccccc}
			\toprule
			$n$ & \multicolumn{2}{c}{MATLAB's \texttt{eig+mrdivide}} & \multicolumn{4}{c}{Our fast algorithm} \\
			\cmidrule(lr){2-3} \cmidrule(lr){4-7}
			& CPU (s) & $\omega_{\text{eig}}$ & Iter & CPU (s) & $\omega_{\text{fast}}$ & $\eta_{\text{fast}} $ \\
			\midrule
			32    &0.002   &$2.40 \times 10^{-14}$   &8   &0.003   & $2.30 \times 10^{-13}$ &$2.56 \times 10^{-14}$\\
			64    & 0.003  & $1.34 \times 10^{-13}$ & 9  & 0.003  & $1.20 \times 10^{-11}$ & $1.02 \times 10^{-14}$ \\
			128   & 0.011  & $1.10 \times 10^{-12}$ & 10  & 0.005  & $3.62 \times 10^{-11}$ & $3.56 \times 10^{-14}$ \\
			256   & 0.050  & $7.37 \times 10^{-12}$ & 10  & 0.023  & $3.47 \times 10^{-10}$ & $2.45 \times 10^{-13}$ \\
			512   & 0.254  & $3.32 \times 10^{-11}$ & 11  & 0.076  & $2.20 \times 10^{-9}$ & $1.42 \times 10^{-12}$ \\
			1024  & 1.285  & $2.00 \times 10^{-10}$ & 12  & 0.306  & $1.90 \times 10^{-8}$ & $7.49 \times 10^{-12}$ \\
			2048  & 7.981  & $1.18 \times 10^{-9}$ & 13  & 1.496  & $6.16 \times 10^{-7}$ & $1.65 \times 10^{-11}$ \\
			4096  & 67.599 & $9.00 \times 10^{-9}$ & 13 & 8.626  & $9.65 \times 10^{-6}$ & $1.17 \times 10^{-10}$ \\
			8192  & 580.663 & $7.56 \times 10^{-8}$ & 14 & 62.270 & $1.16 \times 10^{-4}$ & $1.69 \times 10^{-9}$ \\
			\bottomrule
		\end{tabular}
	\end{table}
	\begin{equation*}
		\eta _{fast\,\,}:=\frac{\left\| D_{eig}-D_{fast} \right\| _F}{\left\| D_{eig} \right\| _F}
	\end{equation*}
	and 
	\begin{equation*}
		w_{eig\,\,}:=\frac{\left\| B-V_{eig}D_{eig}V_{eig}^{-1} \right\| _F}{\left\| B \right\| _F}, \quad w_{fast}\,\,:=\frac{\left\| B-V_{fast}D_{fast}V_{fast}^{-1} \right\| _F}{\left\| B \right\| _F}.
	\end{equation*}
	
	\cref{tab:comparison} shows the comparison of CPU time between using the $\texttt{eig}$ and our proposed fast algorithm for spectral decomposition. The CPU time is estimated by the tic/toc function in MATLAB. The Iter column represents the number of Newton iterations required to reach the tolerance  $tol=10^{-10}$. The CPU time growth of our fast algorithm is much lower than that of the $\texttt{eig}$. Especially when  $n = 8192$, the time required by the $\texttt{eig}$ is approximately 10 times that of the fast algorithm. Considering the effects of discretization errors and rounding errors, \cref{tab:comparison} shows that due to the condition number of the eigenvector matrix  $V$, the results of  $w_{fast}$  are not ideal. However, for $\eta _{fast}$, the eigenvalues obtained by the two algorithms are basically the same.
	\subsection{Numerical performance}
	\label{sec5.2}
	In this section, we present four numerical experiments for the second-order and third-order cases respectively to verify the efficiency of our proposed direct PinT algorithm. Additionally, the examples used in this section are from \cite{BRATSOS2007251,2019fast}.
	
	First, we present two numerical examples of second-order time-dependent differential equations, and compare their acceleration times under different numbers of cores (denoted as $s$).
	
	\noindent\textbf{Example 1.}(\cite{2019fast}) In this example, we consider the two-dimensional nonlinear equation with the Dirichlet boundary conditions:
	\begin{equation}
		\begin{cases}
			\frac{\partial ^2u}{\partial t^2}+\frac{1}{4}\frac{\partial u}{\partial t}=\varDelta u+u-u^2+f\left( x,y,t \right) , \quad  \left( x,y \right) \in \Omega=[0, 2 \pi]^2 ,  t\in \left( 0,T=1 \right] ,\\
			u\left( t,x,y \right) =0,~~\mathrm{on}\,\,\partial \varOmega \times \left[ 0,T \right] ,\\
			u\left( 0,x,y \right) =\sin \left( x \right) \sin \left( y \right) ,  \quad \left( x,y \right) \in \varOmega ,\\
			u_t\left( 0,x,y \right) =0,   \quad   \left( x,y \right) \in \varOmega.
		\end{cases}
		\label{eq5.1:energy}
	\end{equation}
	The exact solution is chosen to be 
	\begin{equation*}
		u\left( x,y,t \right) =\cos \left( t \right) \sin \left( x \right) \sin \left( y \right),
	\end{equation*}
	and the source term $f\left( x,y,t \right) $ is determined correspondingly. Using a central finite difference to approximate $\varDelta $, and approximating with a uniform mesh size  $h$  in both  $x$ and  $y$ directions, the following ODE system is obtained:
	\begin{equation}
		\boldsymbol{u}_{h}^{\prime \prime}\left( t \right) +\frac{1}{4}\boldsymbol{u}_{h}^{\prime}=\varDelta _h\boldsymbol{u}_h\left( t \right) +\boldsymbol{u}_h\left( t \right) -\boldsymbol{u}_{h}^{2}\left( t \right) +\boldsymbol{f}_h\left( t \right) , \quad \boldsymbol{u}_h\left( 0 \right) =\boldsymbol{u}_{0,h},
		\label{eq5.2:energy}
	\end{equation}
	where $\varDelta _h\in \mathbb{R} ^{m\times m}$ is the 5-point stencil Laplacian matrix, $\boldsymbol{u}_h, \boldsymbol{f}_h, \boldsymbol{u}_{0,h}$ represent the finite difference approximation of $u, f, u_0$ corresponding to $m$ spatial grid points. \cref{fig1:example} shows the CPU time requirement of our algorithm to run as $N_t$ gradually increases under different numbers of cores. We can see that the larger $N_t$ is, the more obvious the acceleration effect. When $N_t=2^{9}$, the CPU time corresponding to the number of cores $s$=20 is much smaller than that of $s$=1. It can be seen that as the number of cores continues to increase, the acceleration effect gradually decreases, a possible reason is that as the number of cores increases, the communication time may also increase. In addition, due to computer limitations, these experiments used a maximum of 20 cores for acceleration.
	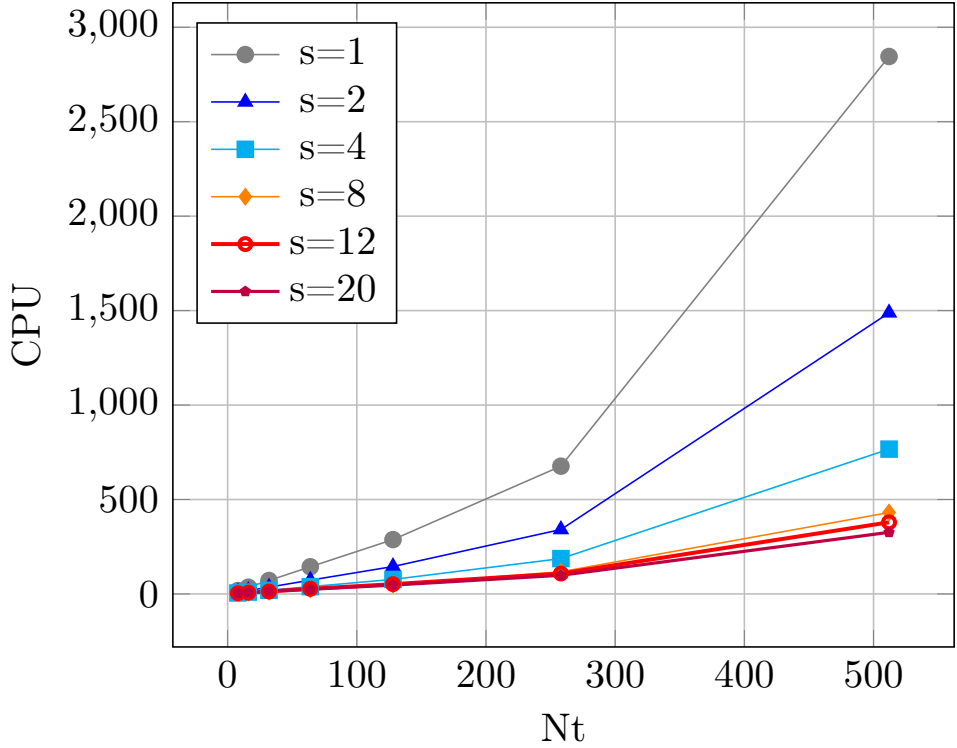
\begin{figure}[h!]
		\centering
		\begin{tikzpicture}[scale=1]
			\begin{axis}[
				xlabel=$N_t$,
				ylabel=CPU,
				xtick={0,100,200,300,400,500,600},
				ytick={0,500,1000,1500,2000,2500,3000},
				%	ymin=0, ymax=3000, % 根据数据范围调整y轴范围
				legend pos=north west, % 图例位置
				grid=both, % 显示网格
				]
				\addplot[
				color=gray,
				mark=*, % 标记点形状
				]
				coordinates {
					(8, 17.895)
					(16, 35.734)
					(32, 71.096)
					(64, 143.809)
					(128, 287.686)
					(258,676.43)
					(512,2845.252) % 这里的x值0,1,2对应xticklabels的顺序，y值根据数据调整
				};
				\addlegendentry{s=1} % 图例标签
				
				\addplot[
				color=blue,
				mark=triangle*,
				]
				coordinates {
					(8, 9.276)
					(16, 18.046)
					(32, 35.962)
					(64, 72.908)
					(128, 145.102)
					(258,341.106)
					(512,1488.17)
				};
				\addlegendentry{s=2}
				
				\addplot[
				color=cyan,
				mark=square*,
				]
				coordinates {
					(8, 5.161)
					(16, 9.874)
					(32, 19.444)
					(64, 38.644)
					(128, 76.962)
					(258,187.129)
					(512,766.672)
				};
				\addlegendentry{s=4}
				
				\addplot[
				color=orange,
				mark=diamond*,
				]
				coordinates {
					(8, 3.605)
					(16, 6.734)
					(32, 13.017)
					(64, 24.974)
					(128, 49.043)
					(258,117.326)
					(512,430.952)
				};
				\addlegendentry{s=8}
				
				\addplot[
				color=red,
				mark=o,
				line width=1pt,
				mark size=1.5pt
				]
				coordinates {
					(8, 3.844)
					(16, 7.039)
					(32, 13.055)
					(64, 27.328)
					(128, 51.879)
					(258,108.009)
					(512,379.307)
				};
				\addlegendentry{s=12}
				
				\addplot[
				color=purple,
				mark=pentagon*, % 这里选择五边形标记点形状，可按需更换
				line width=0.8pt,
				mark size=1pt
				]
				coordinates {
					(8, 3.814)
					(16, 6.564)
					(32, 13.958)
					(64, 24.786)
					(128, 46.845)
					(258,98.677)
					(512,326.123)% 根据实际数据修改坐标值
				};
				\addlegendentry{s=20} 
			\end{axis}
		\end{tikzpicture}
		\caption{Comparison of the CPU time with different number of cores of Example 1, where $m=512^{2}$ and $N_t = 2^3, \cdots, 2^9$.} 
		\label{fig1:example}
	\end{figure}\\
	\noindent \textbf{Example 2.}(\cite{BRATSOS2007251}) In this example, we consider the two-dimensional Sine-Gordon equation with the homogeneous Dirichlet boundary conditions: 
	\begin{equation}
		\begin{cases}
			\frac{\partial ^2u}{\partial t^2}+\rho \frac{\partial u}{\partial t}=\varDelta u-\varphi \left( x,y \right) \sin \left( u \right), \quad  \left( x,y \right) \in \Omega=[0, 2 \pi]^2 , t\in \left( 0,T=1 \right] ,\\
			u\left( t,x,y \right) =0,  ~~\mathrm{on}\,\,\partial \varOmega \times \left[ 0,T \right] ,\\
			u\left( 0,x,y \right) =\sin \left( x \right) \sin \left( y \right), \quad \left( x,y \right) \in \varOmega ,\\
			u_t\left( 0,x,y \right) =0, \quad \left( x,y \right) \in \varOmega ,\\
		\end{cases} 
		\label{eq5.3:energy}
	\end{equation}
	where $\varDelta$ is the Laplace operator, $\varphi \left( x,y \right)$ is a suitable function, $\rho =-1$. The exact solution is chosen to be 
	\begin{equation*}
		u\left( t,x,y \right) =\cos \left( t \right) \sin \left( x \right) \sin \left( y \right).
	\end{equation*}
	
	Approximating with a uniform mesh size  $h$  in both  $x$ and  $y$ directions, and using the same discretization method in the spatial and temporal directions as in \eqref{eq5.1:energy}, it is transformed into the following ODE system: 
	\begin{equation}
		\boldsymbol{u}_{h}^{\prime \prime}+\rho \boldsymbol{u}_{h}^{\prime}=\varDelta \boldsymbol{u}_h-\boldsymbol{\varphi }_h\sin \left( \boldsymbol{u}_h \right), \quad \boldsymbol{u}_h\left( 0 \right) =\boldsymbol{u}_{0,h}.
		\label{eq5.4:energy}
	\end{equation}
	If we define  $\boldsymbol{f}_h=-\boldsymbol{\varphi }_h\sin \left( \boldsymbol{u}_h \right) $, then transform Eq.~\eqref{eq5.4:energy} into 
	\begin{equation}
		\boldsymbol{u}_{h}^{\prime \prime}+\rho \boldsymbol{u}_{h}^{\prime}=\varDelta \boldsymbol{u}_h+\boldsymbol{f}_h, \quad \boldsymbol{u}_h\left( 0 \right) =\boldsymbol{u}_{0,h}.
		\label{eq5.5:energy}
	\end{equation}
	\begin{figure}[h!]
		\centering
		\begin{tikzpicture}[scale=1]
			\begin{axis}[
				xlabel=$N_t$,
				ylabel=CPU,
				xtick={0,100,200,300,400,500,600},
				ytick={0,500,1000,1500,2000,2500,3000},
				%	ymin=0, ymax=3000, % 根据数据范围调整y轴范围
				legend pos=north west, % 图例位置
				grid=both, % 显示网格
				]
				\addplot[
				color=gray,
				mark=*, % 标记点形状
				]
				coordinates {
					(8, 17.253)
					(16, 34.953)
					(32, 69.684)
					(64, 140.425)
					(128, 281.121)
					(258,662.255)
					(512,2792.492) % 这里的x值0,1,2对应xticklabels的顺序，y值根据数据调整
				};
				\addlegendentry{s=1} % 图例标签
				
				\addplot[
				color=blue,
				mark=triangle*,
				]
				coordinates {
					(8, 9.089)
					(16, 17.96)
					(32, 35.904)
					(64, 71.598)
					(128, 143.813)
					(258,338.095)
					(512,1456.153)
				};
				\addlegendentry{s=2}
				
				\addplot[
				color=cyan,
				mark=square*,
				]
				coordinates {
					(8, 5.111)
					(16, 9.881)
					(32, 19.377)
					(64, 38.689)
					(128, 77.017)
					(258,185.989)
					(512,753.888)
				};
				\addlegendentry{s=4}
				
				\addplot[
				color=orange,
				mark=diamond*,
				]
				coordinates {
					(8, 3.747)
					(16, 6.886)
					(32, 12.757)
					(64, 24.735)
					(128, 48.803)
					(258,114.139)
					(512,424.01)
				};
				\addlegendentry{s=8}
				
				\addplot[
				color=red,
				mark=o,
				line width=1pt,
				mark size=1.5pt
				]
				coordinates {
					(8, 3.695)
					(16, 6.939)
					(32, 12.011)
					(64, 23.488)
					(128, 44.866)
					(258,107.117)
					(512,377.522)
				};
				\addlegendentry{s=12}
				
				\addplot[
				color=purple,
				mark=pentagon*, % 这里选择五边形标记点形状，可按需更换
				line width=0.8pt,
				mark size=1pt
				]
				coordinates {
					(8, 4.066)
					(16, 7.081)
					(32, 14.111)
					(64, 48.772)
					(128, 96.858)
					(258,98.677)
					(512,323.244)% 根据实际数据修改坐标值
				};
				\addlegendentry{s=20} 
			\end{axis}
		\end{tikzpicture}
		\caption{Comparison of the CPU time with different number of cores of Example 2, where $m=512^{2}$ and $N_t = 2^3, \cdots, 2^9$.} 
		\label{fig2:example}
	\end{figure}
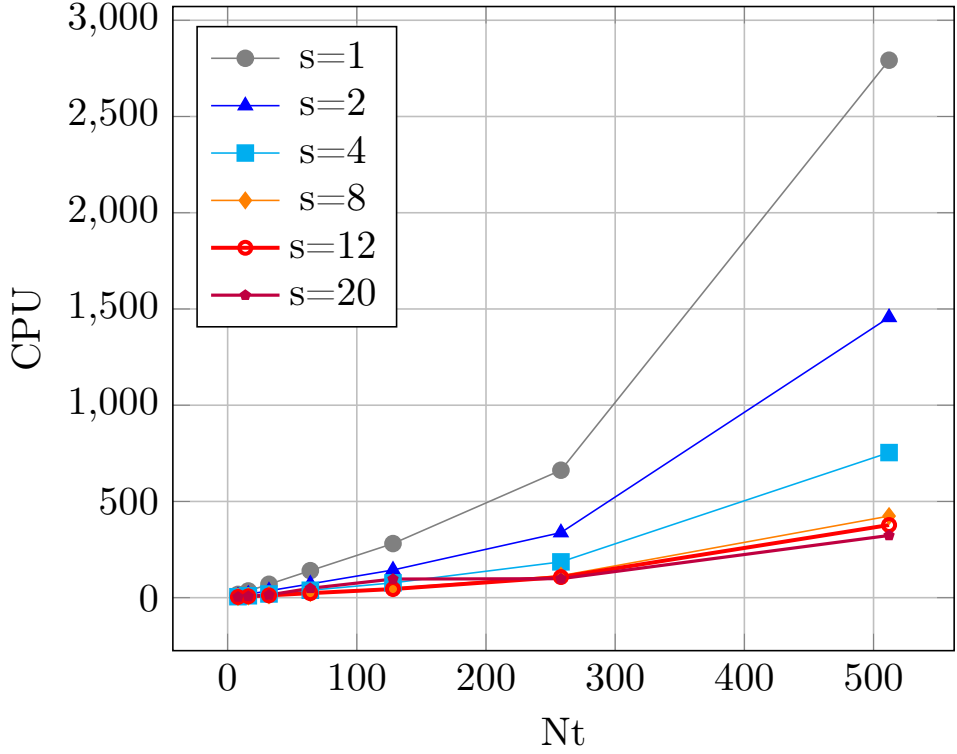
	
	From \cref{fig2:example}, we can see that under different numbers of cores when $N_t$ increases, the acceleration effect becomes more significant. \cref{fig2:example} demonstrates that the curve exhibits progressively flat behavior with the growing number of cores. When $s\geqslant 8$, its speedup effect gradually diminishes. But $N_t=2^9$, the CPU time corresponding to $s=20$ is much smaller than that of $s$=1.
	
	Next, we present two numerical experiments of third-order PDEs, and also compare their computational efficiencies under different numbers of cores.\\
	\textbf{Example 3.}(\cite{2019fast}) In this example, we consider a third-order linear evolution equation with the homogeneous Dirichlet boundary conditions: 
	\begin{equation}
		\begin{cases}
			\frac{\partial ^3u}{\partial t^3}-\frac{\partial u}{\partial t}=\varDelta u-2u+f\left( t,x,y \right) , \quad \left( x,y \right) \in \Omega=[0, 2 \pi]^2 , t\in \left( 0,T=1 \right] ,\\
			u\left( t,x,y \right) =0,  ~~\mathrm{on}\,\,\partial \varOmega \times \left[ 0,T \right] ,\\
			u\left( 0,x,y \right) =\sin \left( x \right) \sin \left( y \right) ,  \quad  \left( x,y \right) \in \varOmega ,\\
			u_t\left( 0,x,y \right) =-2\sin \left( x \right) \sin \left( y \right) , \quad \left( x,y \right) \in \varOmega ,\\
			u_{tt}\left( 0,x,y \right) =4\sin \left( x \right) \sin \left( y \right) , \quad \left( x,y \right) \in \varOmega ,\\
		\end{cases}
		\label{eq5.6:energy}
	\end{equation}
	where $f\left( t,x,y \right) =-2e^{-2t}\sin \left( x \right) \sin \left( y \right)$. The exact solution is given by 
	\begin{equation*}
		u\left( t,x,y \right) =e^{-2t}\sin \left( x \right) \sin \left( y \right).
	\end{equation*}
	
	In the spatial and temporal directions, we use the same discretization scheme as that for the second-order case on a uniform grid. Similar to the second-order scenario, we first transform Eq.~\eqref{eq5.6:energy} into an ODE system, and then apply the proposed third-order one-shot system for calculation, the ODE system is denoted as  
	\begin{equation}
		\boldsymbol{u}_{h}^{\prime \prime \prime}-\boldsymbol{u}_{h}^{\prime}=\varDelta \boldsymbol{u}_h-2\boldsymbol{u}_h+\boldsymbol{f}_h, \quad \boldsymbol{u}_h\left( 0 \right) =\boldsymbol{u}_{0,h}.
		\label{eq5.7:energy}
	\end{equation}
	In \cref{fig3:example}, we can observe that as $N_t$ increases, the CPU grows in a relatively stable manner, the acceleration effect is more obvious when $N_t$ is large. Especially the CPU times corresponding to $s=1$ and $s=20$ when $N_t=2^9$.
	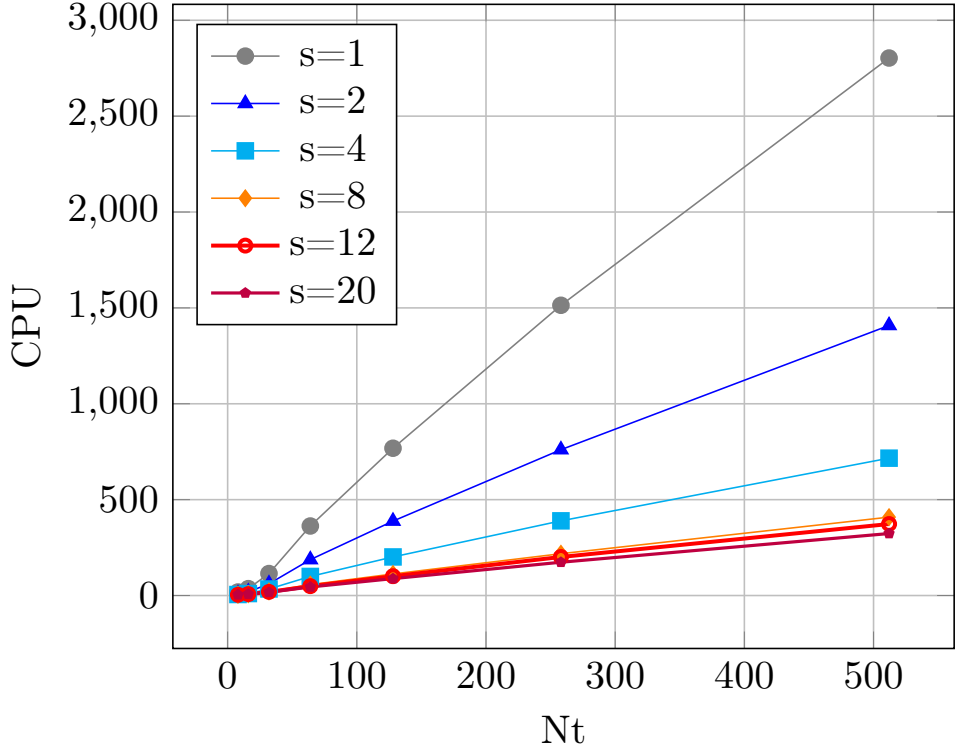
\begin{figure}[h!]
		\centering
		\begin{tikzpicture}[scale=1]
			\begin{axis}[
				xlabel=$N_t$,
				ylabel=CPU,
				xtick={0,100,200,300,400,500,600},
				ytick={0,500,1000,1500,2000,2500,3000},
				%	ymin=0, ymax=3000, % 根据数据范围调整y轴范围
				legend pos=north west, % 图例位置
				grid=both, % 显示网格
				]
				\addplot[
				color=gray,
				mark=*, % 标记点形状
				]
				coordinates {
					(8, 18.378)
					(16, 35.897)
					(32, 114.572)
					(64, 362.582)
					(128, 768.132)
					(258,1513.632)
					(512,2803.1) % 这里的x值0,1,2对应xticklabels的顺序，y值根据数据调整
				};
				\addlegendentry{s=1} % 图例标签
				
				\addplot[
				color=blue,
				mark=triangle*,
				]
				coordinates {
					(8, 9.577)
					(16, 18.548)
					(32, 61.145)
					(64, 187.098)
					(128, 388.075)
					(258,760.466)
					(512,1408.257)
				};
				\addlegendentry{s=2}
				
				\addplot[
				color=cyan,
				mark=square*,
				]
				coordinates {
					(8, 5.208)
					(16, 10.212)
					(32, 34.063)
					(64, 99.5)
					(128, 201.21)
					(258,389.435)
					(512,716.321)
				};
				\addlegendentry{s=4}
				
				\addplot[
				color=orange,
				mark=diamond*,
				]
				coordinates {
					(8, 3.759)
					(16, 7.025)
					(32, 21.314)
					(64, 56.827)
					(128, 112.481)
					(258,218.059)
					(512,408.826)
				};
				\addlegendentry{s=8}
				
				\addplot[
				color=red,
				mark=o,
				line width=1pt,
				mark size=1.5pt
				]
				coordinates {
					(8, 3.973)
					(16, 7.446)
					(32, 18.884)
					(64, 48.884)
					(128, 100.83)
					(258,201.019)
					(512,372.885)
				};
				\addlegendentry{s=12}
				
				\addplot[
				color=purple,
				mark=pentagon*, % 这里选择五边形标记点形状，可按需更换
				line width=0.8pt,
				mark size=1pt
				]
				coordinates {
					(8, 3.917)
					(16, 7.264)
					(32, 16.685)
					(64, 44.631)
					(128, 88.177)
					(258,173.572)
					(512,323.259) % 根据实际数据修改坐标值
				};
				\addlegendentry{s=20} 
			\end{axis}
		\end{tikzpicture}
		\caption{Comparison of the CPU time with different number of cores of Example 3, where $m=512^{2}$ and $N_t = 2^3, \cdots, 2^9$.} 
		\label{fig3:example}
	\end{figure}\\
	\noindent\textbf{Example 4.}(\cite{2019fast}) In this example, we consider the third-order nonlinear evolution equation:  
	\begin{equation}
		\begin{cases}
			\frac{\partial}{\partial t}\left( \frac{\partial ^2u}{\partial t^2}-\varDelta u \right) +\frac{\partial ^2u}{\partial t^2}-D\varDelta u=\beta \varDelta \left( u^2 \right) p\left( t,x,y \right) , \quad \left( x,y \right) \in \Omega=[0, 2 \pi]^2 , t\in \left( 0,T=1 \right] ,\\
			u\left( t,x,y \right) =0, ~~\mathrm{on}\,\,\partial \varOmega \times \left[ 0,T \right] ,\\
			u\left( 0,x,y \right) =\sin \left( x \right) \sin \left( y \right) , \quad \left( x,y \right) \in \varOmega,\\
			u_t\left( 0,x,y \right) =0, \quad \left( x,y \right) \in \varOmega,\\
			u_{tt}\left( 0,x,y \right) =-\sin \left( x \right) \sin \left( y \right) , \quad \left( x,y \right) \in \varOmega,\\
		\end{cases} 
		\label{eq5.8:energy}
	\end{equation}
	with the homogeneous Dirichlet boundary conditions, where $D=\frac{1}{2}$, $\beta$ is a suitable dimensionless constant. The exact solution is given by 
	\begin{equation*}
		u=\cos \left( t \right) \sin \left( x \right) \sin \left( y \right).
	\end{equation*}
	Define  $f\left( t,x,y \right) =\beta \varDelta \left( u^2 \right) p\left( t,x,y \right) $, so that similar to \eqref{eq5.6:energy}, transform \eqref{eq5.8:energy} into the following ODE system, namely  
	\begin{equation}
		\boldsymbol{u}_{h}^{\prime}\left( \boldsymbol{u}_{h}^{\prime \prime}-\varDelta \boldsymbol{u}_h \right) +\boldsymbol{u}_{h}^{\prime \prime}-D\varDelta \boldsymbol{u}_h=\boldsymbol{f}_h, \quad \boldsymbol{u}_h\left( 0 \right) =\boldsymbol{u}_{0,h}.
		\label{eq5.9:energy}
	\end{equation}
	In \cref{fig4:example}, when $N_t$ is large, the CPU with $s=20$ is significantly smaller than that with $s=1$.
	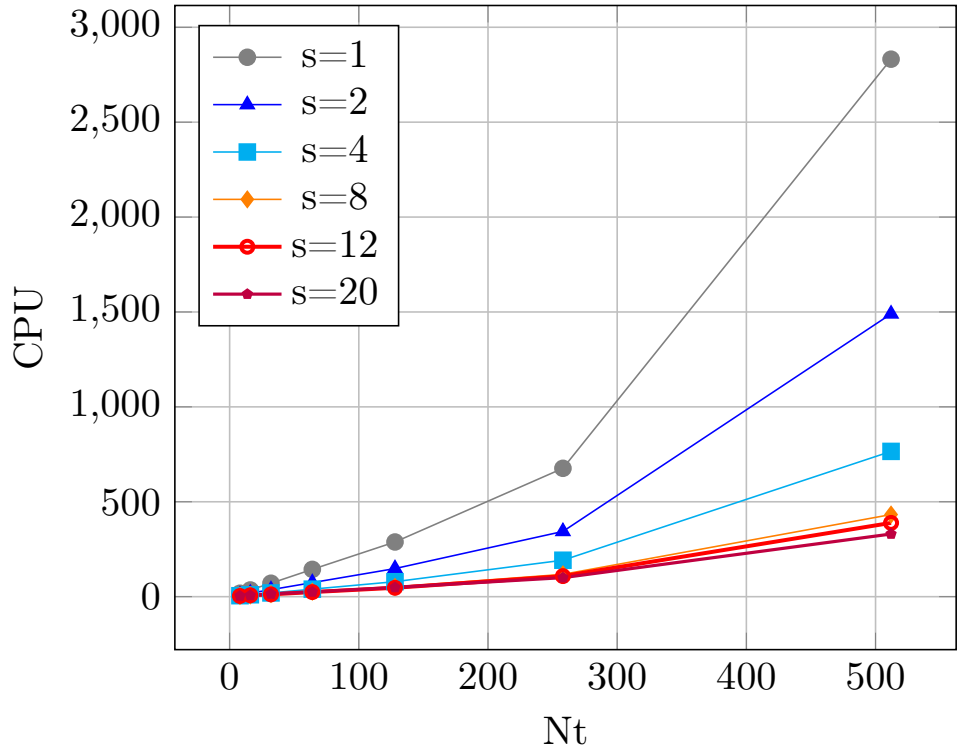
\begin{figure}[h!]
		\centering
		\begin{tikzpicture}[scale=1]
			\begin{axis}[
				xlabel=$N_t$,
				ylabel=CPU,
				xtick={0,100,200,300,400,500,600},
				ytick={0,500,1000,1500,2000,2500,3000},
				%	ymin=0, ymax=3000, % 根据数据范围调整y轴范围
				legend pos=north west, % 图例位置
				grid=both, % 显示网格
				]
				\addplot[
				color=gray,
				mark=*, % 标记点形状
				]
				coordinates {
					(8, 18.207)
					(16, 35.401)
					(32, 70.884)
					(64, 143.876)
					(128, 288.16)
					(258,676.346)
					(512,2831.89) % 这里的x值0,1,2对应xticklabels的顺序，y值根据数据调整
				};
				\addlegendentry{s=1} % 图例标签
				
				\addplot[
				color=blue,
				mark=triangle*,
				]
				coordinates {
					(8, 9.446)
					(16, 18.552)
					(32, 36.788)
					(64, 74.017)
					(128, 148.094)
					(258,343.689)
					(512,1489.782)
				};
				\addlegendentry{s=2}
				
				\addplot[
				color=cyan,
				mark=square*,
				]
				coordinates {
					(8, 5.148)
					(16, 10.053)
					(32, 19.662)
					(64, 39.618)
					(128, 78.975)
					(258,192.087)
					(512,765.478)
				};
				\addlegendentry{s=4}
				
				\addplot[
				color=orange,
				mark=diamond*,
				]
				coordinates {
					(8, 3.894)
					(16, 7.052)
					(32, 13.714)
					(64, 26.358)
					(128, 51.454)
					(258,117.704)
					(512,432.659)
				};
				\addlegendentry{s=8}
				
				\addplot[
				color=red,
				mark=o,
				line width=1pt,
				mark size=1.5pt
				]
				coordinates {
					(8, 3.796)
					(16, 7.24)
					(32, 12.193)
					(64, 24.267)
					(128, 46.74)
					(258,108.434)
					(512,388.709)
				};
				\addlegendentry{s=12}
				
				\addplot[
				color=purple,
				mark=pentagon*, % 这里选择五边形标记点形状，可按需更换
				line width=0.8pt,
				mark size=1pt
				]
				coordinates {
					(8, 3.885)
					(16, 6.801)
					(32, 13.905)
					(64, 24.888)
					(128, 48.816)
					(258,100.209)
					(512,330.115) % 根据实际数据修改坐标值
				};
				\addlegendentry{s=20} 
			\end{axis}
		\end{tikzpicture}
		\caption{Comparison of the CPU time with different number of cores of Example 4, where $m=512^{2}$ and $N_t = 2^3, \cdots, 2^9$.} 
		\label{fig4:example}
	\end{figure}
	\section{Conclusions}
	\label{sec6}
	In this paper, inspired by the technique in \cite{2022A}, we generalize it and apply the diagonalization-based direct time-parallel algorithm to solve time-dependent differential equations of orders 1 to 3. We derive the all-at-once systems for  $\boldsymbol{u}$  of orders 1 to 3 and provide explicit formulas for diagonalization of the time-discrete matrix $B$, that is, $B=VDV^{-1}$. We also prove the condition number  of the eigenvector matrix  $V$  corresponding to $B$. We present numerical experiments for the fast spectral decomposition algorithm. The results demonstrate that our fast algorithm is significantly faster than the MATLAB built-in function $\texttt{eig}$. Moreover, we also design a fast algorithm for the eigenvector matrix $V^{-1}$, which can effectively accelerate the computational efficiency when applied to the fast spectral decomposition algorithm of $B$. Finally, several examples are provided to illustrate the feasibility of our all-at-once system. This work provides a foundation for advancing research in wave equations, fluid dynamics, and related disciplines. For the algorithm proposed in this paper to compute $V^{-1}$, future work could further investigate this method, as it has the potential to significantly accelerate the computation of matrix inversion. 
	
	\section*{CRediT authorship contribution statement}
	\textbf{Shun-Zhi Zhong:} Formal analysis, Investigation, Methodology, Software, Validation, Visualization, Writing – original draft, Writing – review \& editing.
	\textbf{Yong-Liang Zhao:} Conceptualization, Methodology, Project administration, Supervision, Writing – review \& editing.
	\textbf{Qian-Yu Shu:} Conceptualization, Methodology, Project administration, Writing – review \& editing.
	
	\section*{Acknowledgments}
	\addcontentsline{toc}{section}{Acknowledgments}
	\label{sec7}
	\textit{This work is supported by the National Natural Science Foundation of China (12401536), and Sichuan Science and Technology Program (2024NSFSC0441).}
	
	\section*{Data availability}
	The corresponding codes will be publicly available once our work is published.
	
	\section*{Conflict of interests}
   The authors declare that they have no known competing financial interests or personal relationships that could have appeared to influence the work reported in this paper.

%\linenumbers
\bibliographystyle{elsarticle-num} 
\bibliography{Ref}
\end{document}